\documentclass[12pt]{article}%
\usepackage[T1]{fontenc}
\usepackage[applemac]{inputenc}
\usepackage{amsmath,amssymb,fullpage}
\usepackage{amsthm}
\usepackage{amsfonts}
\usepackage[applemac]{inputenc}
\usepackage{color, colortbl}
\usepackage{mathtools}
\usepackage{graphicx}
\usepackage{bm}
\usepackage{tikz}
\usepackage{bbm}
\usepackage{enumerate}
\usepackage{geometry}
\usepackage{hyperref}
\usetikzlibrary{arrows}
\usepackage{graphicx}
\usepackage{subcaption}
\usepackage{caption}

\usepackage{amsfonts}
\usepackage[applemac]{inputenc}
\usepackage{amsmath,amssymb,fullpage}
\usepackage{graphicx}
\usepackage{tikz}
\usepackage{graphicx}
\graphicspath{ {./images/} }
\usepackage{amsmath}
\usepackage{algpseudocode}
\usepackage{algorithm}
\usepackage{float}
\usepackage{mathabx}
\allowdisplaybreaks

\usepackage[numbers]{natbib} 

\DeclareMathOperator*{\esssup}{ess\,sup}
\DeclareMathOperator*{\essinf}{ess\,inf}
\usepackage{siunitx}
\usepackage{placeins}

\let\Oldsection\section
\renewcommand{\section}{\FloatBarrier\Oldsection}

\let\Oldsubsection\subsection
\renewcommand{\subsection}{\FloatBarrier\Oldsubsection}

\let\Oldsubsubsection\subsubsection
\renewcommand{\subsubsection}{\FloatBarrier\Oldsubsubsection}

\usepackage{multirow}
\usepackage{caption}
\usepackage{subcaption}
\usepackage{bbm, dsfont}
\usepackage{amsmath}
\usepackage{verbatim}

\usepackage{calc}
\usepackage{accents}

\newcommand{\vardbtilde}[1]{\widetilde{\raisebox{0pt}[0.9\height]{$\widetilde{#1}$}}}

\algdef{SE}[SUBALG]{Indent}{EndIndent}{}{\algorithmicend\ }%
\algtext*{Indent}
\algtext*{EndIndent}

\newtheorem{theorem}{Theorem}[section]
\newtheorem{lemma}[theorem]{Lemma}

\newtheorem{proposition}[theorem]{Proposition}
\newtheorem{corollary}[theorem]{Corollary}
\newtheorem{definition}[theorem]{Definition}

\newtheorem{assumption}[theorem]{Assumption}

\newtheorem{remark}[theorem]{Remark}

\usepackage{tikz}
\usetikzlibrary{shapes.geometric, arrows}

\tikzstyle{input} = [circle, minimum width=1cm, text centered, draw=black, fill=green!20]

\tikzstyle{output} = [circle, minimum width=1cm, text centered, draw=black, fill=blue!20]

\tikzstyle{lstm} = [rectangle, rounded corners, minimum width=2cm, minimum height=1cm,text centered, draw=black, fill=red!20]

 \tikzstyle{lin} = [rectangle, minimum width=2cm, minimum height=1cm,text centered, draw=black, fill=orange!20]

 \tikzstyle{act} = [ellipse, minimum width=2cm, minimum height=1cm,text centered, draw=black, fill=yellow!20]

\tikzstyle{dot} = [rectangle, minimum width=2cm, minimum height=1cm,text centered]

\tikzstyle{arrow} = [thick,->,>=stealth]
\tikzstyle{map} = [thick, dashed,->,>=stealth]

\usetikzlibrary{arrows,shapes,automata,petri}

\numberwithin{equation}{section}

\newenvironment{keywords}{}{}

\newcommand{\NN}{\mathcal{N}\mathcal{N}}
\newcommand{\bR}{\mathbb{R}}
\newcommand{\bE}{\mathbb{E}}

\begin{document}
\title{Deep Learning for Energy Market Contracts: Dynkin Game with Doubly RBSDEs}
       
\author{ Nacira Agram $^{1,3}$, Ihsan Arharas $^{2}$, Giulia Pucci $^{1}$ and Jan Rems $^{3}$}
\maketitle

\footnotetext[1]{Department of Mathematics, KTH Royal Institute of Technology 100 44, Stockholm, Sweden. 
Email: nacira@kth.se, pucci@kth.se. Work supported by the Swedish Research Council grant (2020-04697).}
\footnotetext[2]{Department of Mathematics, Linnaeus University
(LNU), V\" axj\" o, Sweden. Email: ihsan.arharas@lnu.se }
\footnotetext[3]{Department of Mathematics, University of Ljubljana, Ljubljana, Slovenia. Email: jan.rems@fmf.uni-lj.si. Work supported by Slovenian Research and Innovation Agency, research core funding No.P1-0448.}

\begin{abstract}
We formulate a Contract for Difference (CfD) with early exit options as a two-player zero-sum Dynkin game, reflecting the strategic interaction between an electricity producer and a regulatory entity. The game incorporates penalties for early termination and mean-reverting price dynamics, with the value characterized through a doubly reflected backward stochastic differential equation (DRBSDE).

To compute the contract value and optimal stopping strategies, we develop a neural solver that approximates the DRBSDE solution using a sequence of neural networks trained on simulated trajectories. The method avoids discretizing the state space, supports time-dependent barriers, and scales to high-dimensional settings. We establish a convergence result and test the method on two scenarios: a benchmark symmetric game in 20 dimensions, and a CfD model with 24-dimensional electricity prices representing multiple European zones.

The results demonstrate that the proposed solver accurately captures the contract's value and optimal stopping regions, with consistent performance across dimensional settings.
\end{abstract}



\section{Introduction}

Electricity markets operate in a dynamic environment where prices fluctuate due to changes in supply and demand, fuel costs, regulatory interventions, and technological developments. The shift to a low carbon energy system has increased reliance on intermittent renewable sources such as wind and solar. While environmentally sustainable, these technologies introduce revenue uncertainty due to volatile output and fluctuating market prices, posing financial challenges for producers. In response, governments and regulatory bodies have introduced financial instruments to de-risk investments in green electricity.

A central mechanism in this transition is the \emph{Contract for Difference} (CfD), originally developed in financial markets and now widely adopted in European energy policy, especially following the 2023 electricity market reform \cite{eu2023}. Under a CfD, a producer receives a fixed strike price for electricity sold: if the market price falls below the strike, the regulator compensates the shortfall; if the price exceeds the strike, the producer returns the surplus. This arrangement stabilizes producer revenues while maintaining market-based incentives.

However, most CfDs include \emph{early exit clauses}, allowing either party to terminate the contract before maturity by paying a penalty. This optionality introduces a strategic interaction between producer and regulator: both parties evaluate when to exit based on market conditions and financial risk. We formulate this setting as a \emph{two-player zero-sum Dynkin game}, where each player chooses a stopping time to minimize or maximize expected payoff that includes a running cost, penalties, and a terminal reward.

Precisely, let $T > 0$ be a fixed time horizon, and consider a starting time $t \in [0,T]$. Consider a filtered probability space $(\Omega, \mathcal{F}, \{\mathcal{F}_t\}_{t \in [0,T]}, \mathbb{P})$ satisfying the usual conditions and supporting a $d$-dimensional Wiener process $B = (B_t)_{t \in [0,T]}$, where $d \ge 1$. Let $\mathcal{T}_{t,T}$ denote the set of all stopping times $\tau$ such that $t \leq \tau \leq T$ almost surely with respect to the filtration $\{\mathcal{F}_t\}$. To describe the underlying market dynamics, we model electricity prices as a multi-dimensional stochastic process \( X^{t,x} = (X_s^{t,x})_{s \in [t,T]} \) governed by the \emph{Stochastic Differential Equation} (SDE):
\begin{equation} \label{eq:SDE-general}
    dX^{t,x}_s = b(s, X^{t,x}_s)\, ds + \sigma\, (s, X^{t,x}_s) dB_s, \quad X_t^{t,x} = x \in \mathbb{R}^d,
\end{equation}
where $b : [0, T] \times \mathbb{R}^d \to \mathbb{R}^d$ and $\sigma : [0, T] \times \mathbb{R}^d \to \mathbb{R}^{d \times d}$ are measurable functions.

The payoff functional associated with the Dynkin game is given by
\begin{align} \label{dynkin_game}
J_{t,x}(\tau_1, \tau_2) &= \mathbb{E}\Bigg[ \int_t^{\tau_1 \wedge \tau_2} \varphi(s, X_s^{t,x})\, ds 
+ f_1(\tau_1, X_{\tau_1}^{t,x}) \mathbbm{1}_{\{\tau_1 \le \tau_2, \tau_1 < T\}} \notag \\
&\qquad - f_2(\tau_2, X_{\tau_2}^{t,x}) \mathbbm{1}_{\{\tau_2 < \tau_1\}} 
+ g(X_T^{t,x}) \mathbbm{1}_{\{\tau_1 \wedge \tau_2 = T\}} \Big| \mathcal{F}_t \Bigg],
\end{align}
and the game is considered ``fair'' if the upper and lower values for the game coincide:
\begin{equation} 
    \essinf_{\tau_1 \in \mathcal{T}_{t, T}} \esssup_{\tau_2 \in \mathcal{T}_{t, T}} J_{t,x}(\tau_1, \tau_2) 
    = V(t,x) = \esssup_{\tau_2 \in \mathcal{T}_{t, T}} \essinf_{\tau_1 \in \mathcal{T}_{t, T}} J_{t,x}(\tau_1, \tau_2).
\end{equation}
The common value $ V(t,x) $ represents the value of the game. This Dynkin game is equivalently characterized by a \emph{Doubly Reflected Backward Stochastic Differential Equation} (DRBSDE), which takes the form:
\begin{equation} \label{eq:drbsde_intro}
Y_s^{t,x} = g(X_T^{t,x}) + \int_s^T \varphi(r, X_r^{t,x})\, dr - \int_s^T Z_r^{t,x}\, dB_r + (A_T^{t,x} - A_s^{t,x}) - (C_T^{t,x} - C_s^{t,x}),
\end{equation}
subject to the reflection constraints:
\[
-f_2(s, X_s^{t,x}) \le Y_s^{t,x} \le f_1(s, X_s^{t,x}), \quad \forall s \in [t,T],
\]
with Skorokhod minimality conditions ensuring that \( Y^{t,x} \) is reflected only when necessary. This formulation was introduced by \cite{cvitanic_karatzas_1996} and extended to broader settings in, e.g., \cite{hamadene_2006, dumitrescu2017, Ouknine2024, chassagneux2009discrete, dumitrescu2016reflected}.

 When modeling the price dynamics underlying a CfD, we adopt the \emph{Ornstein-Uhlenbeck} (OU) process, a widely used model in energy finance due to its capacity to capture mean-reverting behavior, its empirical calibration tractability, and its analytical convenience \cite{lucia2002electricity, benth2008stochastic}.

We interpret the components of the payoff functional \eqref{dynkin_game} as follows: Player~1, representing the regulatory authority, incurs a penalty \( f_1(\tau_1, X^{t,x}_{\tau_1}) \) upon early termination, while Player~2, the electricity generator, pays a penalty \( f_2(\tau_2, X^{t,x}_{\tau_2}) \) if they choose to exit before the contract's maturity. If neither party exits early, the contract settles at the terminal time \( T \) without further adjustment. These asymmetric incentives define a two-player zero-sum Dynkin game, where each player strategically selects an optimal stopping time based on the stochastic evolution of electricity prices. The fair value of the contract is thus identified with the value of the corresponding Dynkin game between the two players.

Although the CfD model we study is Markovian, we deliberately adopt a DRBSDE formulation for two key reasons. First, DRBSDEs provide a robust probabilistic representation of Dynkin games that can handle irregular or non-Markovian settings, and are particularly well suited for backward in time numerical algorithms. This structural advantage aligns with the interpretation proposed by \cite{dumitrescu2017}, who emphasize that even in Markovian contexts, backward stochastic formulations offer flexible characterizations and solver friendly implementations.

Second, the DRBSDE approach naturally integrates with deep learning methods that rely on trajectory simulation and backward training. Unlike PDE methods, which require fine spatial discretization and often become computationally infeasible in high dimensions, DRBSDEs enable mesh-free solvers that scale efficiently to complex energy market models involving multiple regions, products, or time scales. This perspective is reinforced by the numerical methods developed in \cite{chassagneux2009discrete} and \cite{dumitrescu2016reflected}, as well as the recent survey by \cite{chessari2023survey}. Furthermore, recent theoretical advances have extended DRBSDEs to non-Markovian and rough settings, such as $G$-Brownian motion \cite{li2025doubly}, and have analyzed Nash equilibria in reflected game settings \cite{lin2013nash}.

To the best of our knowledge, we are the first to introduce a backward deep learning algorithm, referred to as the \emph{Deep DRBSDE solver}, to compute the solution of the associated DRBSDE by learning both the value process \( Y \) and the optimal stopping regions from sample trajectories. Our solver accommodates high-dimensional dynamics.

\paragraph{Paper Organization.}
Section~\ref{sec:application} formulates the CfD with early exit options as a two-player Dynkin game and describes its representation via a DRBSDE. In Section~\ref{sec:deep_learning}, we develop a neural network algorithm for solving the DRBSDE in high-dimensional settings. Section~\ref{sec:conv} provides a convergence analysis of the proposed method. Finally, Section~\ref{sec:num experiments} presents numerical results, consisting of a 20-dimensional benchmark problem and a CfD model with 24-dimensional electricity prices. 
The Python implementation of the Deep DRBSDE solver, along with all numerical experiments, is available at:
\url{https://github.com/giuliapucci98/DRBSDE-Dynkin-Game}.

\section{Contract for Differences with Exit Options in En-
ergy Markets}
\label{sec:application}
In this section, we present a stochastic game-theoretic framework for two-way CfD in electricity markets, incorporating early termination options for both the regulatory entity and the power producer, denoted also as Player $1$ and Player $2$, respectively. \\

CfDs are financial instruments widely used to stabilize revenues for electricity generators while providing predictable cost structures for regulators. These contracts define a fixed strike price: if the market price falls below this level, the regulator compensates the generator; if the price exceeds it, the surplus is returned. Such mechanisms are central to reducing price volatility and supporting long-term investments in renewable energy infrastructure \cite{ADP, BGJWK}. \\

To enhance contractual flexibility, modern CfDs can include \emph{two-way exit options}, allowing either party to terminate the agreement before its maturity.  Early termination triggers penalty payments, introducing a strategic element: each party must weigh the cost of exiting against the expected value of remaining in the contract. This naturally leads to a \emph{two player Dynkin game}, in which both participants seek to optimize their respective outcomes by choosing an optimal stopping time. We adopt this game-theoretic framework and link it to a DRBSDE, which characterizes the value function of the game and the equilibrium stopping strategies. Theoretical background on DRBSDEs and their connection to Dynkin games is reviewed in Appendix~\ref{appendix:drbsde}. \\




\subsection{Market Price Modeling}\label{sec:market-price}

We model electricity prices using a $d$-dimensional OU process, which captures mean-reverting behavior observed in energy markets. The forward process \(X^{t,x} = (X_s^{t,x})_{s \in [t,T]}\) evolves as:
\begin{equation}
    dX^{t,x}_s = \kappa(\mu - X^{t,x}_s)\,ds + \sigma\,dB_s, \quad X^{t,x}_t = x \in \mathbb{R}^d,
    \label{eq:OU}
\end{equation}
where:
\(\kappa \in \mathbb{R}^{d \times d}\) is the mean-reversion rate matrix which is positive definite,
 \(\mu \in \mathbb{R}^d\) is the long term mean level,
\(\sigma \in \mathbb{R}^{d \times d}\) is the volatility matrix and \(B = (B^1, \dots, B^d)\) is a $d$-dimensional standard Brownian motion.

\vspace{0.5em}
The OU process is a widely adopted model in energy finance for describing electricity spot and forward prices; see, e.g., \cite{lucia2002electricity, benth2008stochastic}. Its defining feature is the mean-reverting drift, which reflects the tendency of electricity prices to revert toward a long-term equilibrium due to real time balancing mechanisms and market regulations.\\
From an economic standpoint, electricity is non storable, and supply demand imbalances are rapidly corrected. This leads to transient price deviations and a pullback toward a baseline level, which is naturally captured by the drift term \( \kappa(\mu - X) \) in the OU model.\\
From a statistical standpoint, the OU process is analytically tractable: its transition densities, autocorrelation structure, and moment generating functions are explicitly available. This facilitates calibration to historical data and ensures reliable simulation under both the historical and risk-neutral measures. We leverage these properties in our calibration study in Section~\ref{calibration}.\\
In our numerical framework, the OU process is particularly well suited to DRBSDE solvers. Its continuous sample paths and affine structure simplify both forward simulation and backward approximation. Moreover, the OU dynamics satisfy the regularity assumptions such as Lipschitz continuity of the coefficients and bounded moments that underpin the DRBSDE well-posedness (Theorem~\ref{existence of solution DRBSDE}) and the convergence of the Deep DRBSDE solver (Theorem~\ref{theo:algo}).
\vspace{0.5em}
Importantly, we emphasize that the OU model is used for empirical realism and computational tractability, not due to limitations of our method. Our algorithm and convergence theory apply to a broader class of Markovian state processes with Lipschitz coefficients. Thus, while our numerical experiments focus on the OU model, the solver is readily applicable to more general dynamics.

\subsection{Exit Options and Strategic Decision-Making}
The key idea behind a two-way CfD initiated at $t$ and maturing at $T$ is that, at each time $s \in [t,T]$, the amount exchanged between the parties is adjusted according to the weighted difference between the market price vector \(X^{t,x}_s \in \mathbb{R}^d\) and a given fixed strike price vector \(K \in \mathbb{R}^d\). This translates into setting a payoff function of the form
\begin{equation}
\varphi(s, X^{t,x}_s) = \langle w,  \bigl(K - X^{t,x}_s\bigr) e^{-\rho (s - t)} \rangle,\label{phi}
\end{equation}where $w \in \bR^d$ is the weight vector with elements summing to 1, and \(\rho > 0\) is the discount rate. The introduction of the weights allows players to put a different emphasis on different prices that are included in the contract.
To discourage premature termination, penalty clauses for early exits are introduced. Player $1$, representing the regulatory entity, incurs a penalty $f_1(\tau_1,X^{t,x}_{\tau_1})$ upon early termination, while Player $2$, the electricity generator, pays a penalty $f_2(\tau_2,X^{t,x}_{\tau_2})$ if they choose to withdraw before the contract's expiration. 
If neither party exits early, the contract reaches maturity at $T$ and no additional terminal adjustment is required.

These competing incentives create a strategic conflict, naturally leading to a two-player Dynkin game, where each party optimally selects a stopping time $\tau_i, i=1,2$, to maximize their respective payoffs. The expected cost to Player $1$, which is equivalent to the gain of Player $2$, is given by:
\begin{align} \label{cfd payoff}
    J_{t,x}(\tau_1, \tau_2) &= \mathbb{E}\bigg[
        \displaystyle\int_t^{\tau_1 \wedge \tau_2} \varphi(s, X^{t,x}_s) \, ds
        + f_1(\tau_1,X^{t,x}_{\tau_1}) \mathbbm{1}_{\{\tau_1 \leq \tau_2, \tau_1 < T\}} 
        \notag \\
        & - f_2(\tau_2,X^{t,x}_{\tau_2}) \mathbbm{1}_{\{\tau_2 < \tau_1\}}   \bigg| \, \mathcal{F}_t \bigg], \quad \tau_1, \tau_2 \in \mathcal{T}_{t, T}. 
\end{align} 
Here, $\tau_1 \wedge \tau_2 $ denotes the first contract termination. Player $1$ aims to minimize $J_{t,x}(\tau_1, \tau_2)$, while Player $2$ seeks to maximize it, leading to a zero-sum game structure. Moreover, the upper and lower value functions of the game on $[t, T]$ are given by:
\begin{equation*}
    \overline{V}(t,x) = \essinf_{\tau_1 \in \mathcal{T}_{t, T}} \esssup_{\tau_2 \in \mathcal{T}_{t, T}} J_{t,x}(\tau_1, \tau_2), \quad
    \underline{V}(t,x) = \esssup_{\tau_2 \in \mathcal{T}_{t, T}} \essinf_{\tau_1 \in \mathcal{T}_{t, T}} J_{t,x}(\tau_1, \tau_2).
\end{equation*}

The solution to this problem requires determining optimal stopping times that satisfy the equilibrium.

\subsection{Solution via DRBSDEs}

We characterize the value of the two-player Dynkin game \ref{cfd payoff} using a DRBSDE. The connection between Dynkin games and DRBSDEs is well-established in the literature and recalled in Appendix~\ref{sec DRBSDEs}. This formulation allows us to compute both the game's value and the optimal stopping strategies for the two players.\\

We recall below the standard functional spaces used in the formulation and analysis of DRBSDEs:

\begin{itemize}
    \item $\mathcal{S}$ is the space of $\mathcal{F}_t$-adapted continuous processes $(Y_t)_{t \leq T}$ with values in $\mathbb{R}$, and $\mathcal{S}^2 := \{ Y \in \mathcal{S} \mid \mathbb{E}[\sup_{t \leq T} |Y_t|^2] < \infty \}$.
    
    \item $\widetilde{\mathcal{P}}$ (resp. $\mathcal{P}$) denotes the $\mathcal{F}_t$-progressive (resp. predictable) $\sigma$-algebra on $\Omega \times [0, T]$.

    \item $\mathbb{L}^2$ is the space of $\mathcal{F}_T$-measurable random variables $\xi: \Omega \to \mathbb{R}$ such that $\mathbb{E}[|\xi|^2] < \infty$.

    \item $\mathcal{H}^{2,d}$ (resp. $\mathcal{H}^d$) is the space of $\widetilde{\mathcal{P}}$-measurable processes $Z = (Z_t)_{t \leq T}$ with values in $\mathbb{R}^d$ that are square-integrable with respect to $d\mathbb{P} \otimes dt$ (resp. $\mathbb{P}$-a.s. $dt$-square integrable).

    \item $S_{ci}$ (resp. $S_{ci}^2$) denotes the space of continuous, non-decreasing, $\mathcal{P}$-measurable processes $A = (A_t)_{t \leq T}$ with $A_0 = 0$ (resp. and $\mathbb{E}[A_T^2] < \infty$).
\end{itemize}

 Given that the penalty functions $f_1$ and $f_2$ satisfy Assumption \ref{ass 2}, we consider the unique $\mathcal{P}$-measurable solution $(Y^{t,x}, Z^{t,x}, A^{t,x}, C^{t,x})$ of the DRBSDE associated with the data of the two-way CfD: 
\[
\left( \varphi(\cdot, X^{t,x}_\cdot),\, 0,\, f_1(\cdot, X^{t,x}_\cdot),\, -f_2(\cdot, X^{t,x}_\cdot) \right),
\]
that is,
\begin{itemize}
    \item[(i)] $Y^{t,x} \in \mathcal{S}^2, \, Z^{t,x} \in \mathcal{H}^{2,d}, \, A^{t,x} \in S_{ci}, \, \text{and} \, C^{t,x} \in S_{ci}$.    
    \item[(ii)] For each $s \in [t, T]$, 
    \begin{equation} \label{DRBSDE cfd}
    Y^{t,x}_s = \displaystyle\int_s^T \varphi(r,X^{t,x}_r)dr
   - \displaystyle\int_s^T Z^{t,x}_r dB_r + (A^{t,x}_T - A^{t,x}_s) - (C^{t,x}_T - C^{t,x}_s).
    \end{equation}
    \item[(iii)] $ -f_2(s, X^{t,x}_s) \leq Y^{t,x}_s \leq f_1(s, X^{t,x}_s), \quad \forall s \in [t, T].$ 
    \item[(iv)] $
    \displaystyle\int_t^T (Y^{t,x}_r + f_2(r,X^{t,x}_r)) dA^{t,x}_r = 
    \int_t^T (f_1(r,X^{t,x}_r) - Y^{t,x}_r) dC^{t,x}_r = 0.$ 
\end{itemize}
 The existence and uniqueness of the solution follow from Theorem \ref{existence of solution DRBSDE}. \\

The Skorokhod condition (iv) ensure that the processes $A^{t,x}$ and $C^{t,x}$ act only when necessary to keep $Y^{t,x}$ within the barriers $[-f_2, f_1]$. The increasing process $A^{t,x}$ adjusts $Y^{t,x}$ only when it reaches the lower boundary $-f_2$, meaning that Player 1 is forced to stop at this level. Similarly, the process $C^{t,x}$ increases only when $Y^{t,x}$ reaches the upper boundary $f_1$, forcing Player 2 to stop. \\

We also define the associated stopping strategies:
\begin{align*}
    \tau_{1,t}^* &:= \inf \left\{ s \geq t : Y^{t,x}_s = f_1(s, X^{t,x}_s) \right\} \wedge T, \\
    \tau_{2,t}^* &:= \inf \left\{ s \geq t : Y^{t,x}_s = -f_2(s, X^{t,x}_s) \right\} \wedge T.
\end{align*}

Then the first component of the DRBSDE solution satisfies
\begin{equation}
    Y^{t,x}_t = \overline{V}(t,x) = \underline{V}(t,x),
\end{equation}
and the pair $(\tau_{1,t}^*, \tau_{2,t}^*)$ constitutes a saddle point of the Dynkin game. This identification of the game value and equilibrium strategies via the DRBSDE solution corresponds exactly to Theorem~\ref{theo:link}, recalled in Appendix~\ref{sec DRBSDEs}, which formalizes the equivalence between two-player Dynkin games and doubly reflected BSDEs with bilateral barriers.



\subsection{Application to Exponentially Decaying Penalty Structures}

A common approach to modeling penalties for early contract termination assumes that the cost of terminating the contract diminishes over time, reflecting the increasing reluctance of counterparties to withdraw as maturity approaches. This can be captured through an exponentially decaying penalty function:
\begin{equation}
    f_1(s,X^{t,x}_s)= \gamma_1 e^{ -\rho( s-t)}, \quad f_2(s,X^{t,x}_s)= \gamma_2 e^{- \rho (s-t)}.
    \label{barriers}
\end{equation}
This choice aligns with practical penalty structures observed in electricity markets, where long term commitments are encouraged, and early withdrawals are penalized. By explicitly writing the driver function \(\varphi\) defined in \eqref{phi}, the expected outcome in Equation \eqref{cfd payoff} becomes
\begin{align} \label{cfd payoff-explicit}
    J_{t,x}(\tau_1, \tau_2) =& \mathbb{E}\bigg[
        \int_t^{\tau_1 \wedge \tau_2} \langle w,  \bigl(K - X^{t,x}_s\bigr) e^{-\rho (s - t)} \rangle \, ds
        +  \gamma_1 e^{ -\rho( \tau_1-t)} \mathbbm{1}_{\{\tau_1 \leq \tau_2, \tau_1 < T\}} 
        \notag \\
        & - \gamma_2 e^{- \rho (\tau_2-t)} \mathbbm{1}_{\{\tau_2 < \tau_1\}}   \bigg| \, \mathcal{F}_t \bigg], \quad \tau_1, \tau_2 \in \mathcal{T}_{t, T}. 
\end{align} 

To incorporate this structure into the DRBSDE framework, we reformulate the reflected constraints in terms of the exponentially decaying barriers:
\begin{equation}
    \begin{cases}
        Y^{t,x}_s =  \displaystyle\int_s^T\langle w,  \bigl(K - X^{t,x}_r\bigr) e^{-\rho (r - t)} \rangle \ dr - \int_s^T Z^{t,x}_r dB_r + (A^{t,x}_T - A^{t,x}_s) - (C^{t,x}_T - C^{t,x}_s), \\[10pt]
        -\gamma_2 e^{- \rho (s-t)} \leq Y^{t,x}_s \leq \gamma_1 e^{- \rho (s-t)}, \\[10pt]
        \displaystyle\int_t^T (Y^{t,x}_r + \gamma_2 e^{- \rho (r-t)}) dA^{t,x}_r = 
        \int_t^T (\gamma_1 e^{- \rho (r-t)} - Y^{t,x}_r) dC^{t,x}_r = 0.
         \label{DRBSDEdecay}
    \end{cases}
\end{equation}

This formulation ensures that the solution remains within the time dependent barriers given by the decaying penalty functions. The increasing processes $A^{t,x}$ and $C^{t,x}$ act to keep $Y^{t,x}$ inside the interval determined by the exponentially decaying constraints. The effect of this formulation is that early terminations are discouraged more strongly at the beginning of the contract period, whereas termination becomes more feasible as $s$ approaches $T$ due to the vanishing penalty terms.

In this setting, the optimal stopping times $\tau_{1,t}^*$ and $\tau_{2,t}^*$ are adapted to the time-dependent constraints:
\begin{equation*}
    \tau_{1,t}^* = \inf \{s \geq t : Y^{t,x}_s = \gamma_1 e^{-\rho (s-t)} \}, \quad \tau_{2,t}^* = \inf \{s \geq t : Y^{t,x}_s = -\gamma_2 e^{-\rho (s-t)} \}.
\end{equation*}

This penalty structure reflects the dynamics of contract termination in electricity markets, emphasizing the consequences of non compliance. It also allows regulators and market participants to analyze the sensitivity of optimal stopping decisions to penalty decay rates, which can inform policy decisions on contract design and risk mitigation strategies.


\section{Deep Learning for DRBSDEs}\label{sec:deep_learning}

In this section, we extend the deep learning-based algorithm introduced in \cite{pham} for solving reflected BSDEs to the case of DRBSDEs. Our approach employs feedforward neural networks to approximate the unknown functions associated with the DRBSDE. These networks provide an efficient way to learn complex functional relationships through affine transformations and nonlinear activation functions. Moreover, we address the algorithm's convergence in terms of the DRBSDE solution, the optimal stopping times, and the value function, ensuring the effectiveness of our approach.

From now on, we will assume that all contracts are stipulated at time $t=0$. Consequently, we will omit starting time and initial state from the superscripts of the processes. This simplification does not affect the algorithm but simplifies the notation.

\subsection{Neural Network Architecture}

The neural network consists of $L+1$ layers, where $L > 1$, and $N_\ell$ neurons in each layer, for $\ell = 0, \dots, L$. The first layer, known as the input layer, has $N_0$ neurons, corresponding to the dimension of the state variable $x$. The output layer has $N_L$ neurons, while the $L-1$ hidden layers each contain $N_\ell = h$ neurons, for $\ell = 1, \dots, L-1$.

A feedforward neural network is a function mapping $\mathbb{R}^{N_0}$ to $\mathbb{R}^{N_L}$, expressed as:
\begin{equation*}
\mathcal{N}(x;\theta) = (A_L \circ \psi \circ A_{L-1} \circ \psi \circ \dots \circ \psi \circ A_1)(x),
\end{equation*}
where each $A_\ell$ is an affine transformation defined as:
\begin{equation*}
A_\ell(x) = W_\ell x + b_\ell,
\end{equation*}
where $W_\ell \in \mathbb{R}^{N_\ell \times N_{\ell-1}}$ is the weight matrix and $b_\ell \in \mathbb{R}^{N_\ell}$ is the bias vector for layer $\ell$. The activation function $\psi: \mathbb{R} \to \mathbb{R}$ is applied component-wise after each affine transformation. Common choices for activation functions include ReLU, tanh, and sigmoid functions. In the notation $\mathcal{N}(\theta)$, the parameter vector $\theta$ represents all trainable weights and biases in the network.

\subsection{Time Discretization of the Forward and Backward Components} \label{discretization}

To approximate the solution of the DRBSDE, we begin by discretizing both the forward SDE and the backward equation. This discretization provides training data and forms the basis for the learning algorithm.

\vspace{0.2cm}
\noindent\textbf{Forward Process.} 
Let \( N \in \mathbb{N} \) denote the number of time steps and consider the uniform grid \( \pi := \{t_0, t_1, \dots, t_N\} \) on the interval \( [0, T] \), where \( \Delta t = \frac{T}{N} \). The forward SDE \eqref{eq:SDE-general} is discretized using the Euler-Maruyama scheme:
\begin{equation*}
\begin{cases}
X^N_{n+1} = X^N_n + b(t_n, X^N_n)\, \Delta t + \sigma(t_n, X^N_n)\, \Delta B_{n+1}, \\
X^N_0 = x, \quad \Delta B_{n+1} := B_{t_{n+1}} - B_{t_n}, \quad n = 0, \dots, N-1.
\end{cases}
\end{equation*}

\vspace{0.2cm}
\noindent\textbf{Backward Process.} 
If we ignore the terms including processes $A$ and $C$ in Equation \eqref{eq:drbsde_intro}, we obtain the following process 
\begin{align}\label{eq:bem}
    \widetilde{Y}_{t} = Y_{t_{n+1}} +  \displaystyle\int_{t}^{t_{n+1}} \varphi(s, X_s) ds
   - \displaystyle\int_{t}^{t_{n+1}} Z_s dB_s,
\end{align}
defined on each subinterval $[t_n, t_{n+1})$ for $n = 0, \ldots, N-1$, while $\widetilde{Y}_{t_N} = g(X_T)$. This approximation captures the drift and martingale components of the solution, assuming no contact with the barriers on the current sub-interval.

To account for the reflection mechanism, the process \( \widetilde{Y} \) is subsequently projected back onto the interval defined by the barriers \( [-f_2(t, X_t), f_1(t, X_t)] \) at each grid point,
\[
\widehat{Y}_{t_n} = \min\left( \max\left( \widetilde{Y}_{t_n}, -f_2(t_n, X_{t_n}) \right), f_1(t_n, X_{t_n}) \right).
\]
This enforces the constraint structure of the DRBSDE and ensures that the Skorokhod conditions are respected in the discrete time limit.

\vspace{0.2cm}
\noindent\textbf{Gradient-Based Approximation.} 
To approximate \( Y \) and \( Z \), we use the value function representation:
\begin{equation*}
Y_0 = V(0,x), \qquad Z_t = \sigma^\intercal(t, X_t)\, \partial_x V(t, X_t), \quad t \in [0,T].
\end{equation*}
The functions \( V \) and \( \partial_x V \) are parameterized by neural networks and optimized via stochastic gradient descent. The full architecture and learning algorithm are presented in the next section.

\subsection{Neural Network Approximation of the DRBSDE Solution}

To approximate the solution of the DRBSDE, we use a localized algorithm where at each discrete time step $t_n$ we employ two independent neural networks and denote:
\begin{itemize}
    \item $\widetilde{Y}^N_n =  \mathcal{Y}^N_n(t_n, X_n; \theta_n^1)$ to approximate $\widetilde{Y}_{t_n}$, 
    \item $\widehat{Z}^N_n = \mathcal{Z}^N_n(t_n, X_n; \theta_n^2)$ to approximate $Z_{t_n}$.
\end{itemize}
In practice, these two networks are combined into a single larger network, denoted by $\mathcal{NN  }_{n}(t_n, X_n; \theta_n)$, where $\theta_n= (\theta_n^1,\theta_n^2)$ represents the full set of trainable parameters.  The network is trained in a way that ensures the obtained processes follow the dynamics in Equation \eqref{eq:bem}. After the training, we set 
\begin{equation}\label{eq:minmax}
    \widehat{Y}^N_n = \min(\max(\widetilde{Y}^N_{n}, - f_2(t_n, X_n^{N})),f_1(t_n, X^N_{n})) 
\end{equation} 
to ensure that the obtained approximation follows the doubly reflected solution in Equation \eqref{eq:drbsde_intro}.
The steps involved in training the neural networks and solving the DRBSDE are outlined in the following algorithm:

\begin{algorithm}[H]
\caption{Deep DRBSDE solver}
\label{algo:DRBSDE}
\begin{algorithmic}[1]
\For{$n = N-1$ to $0$}
    \For{each epoch}
        \For{$j = 0$ to $M$}
            \State Initialize state $X_0$ with initial condition $x_0$
            \For{$i = 0$ to $n$}
                \State Sample Brownian motion increment $\Delta B_{i+1}^{N,j}$
                \State Compute $X_{i+1}^{N,j} = X_i^{N,j} + b(t_i, X_i^{N,j}) \Delta t + \sigma(t_i, X_i^{N,j}) \Delta B_{i+1}^{N,j}$
            \EndFor
            \State Compute $\widetilde{Y}^{N,j}_n, \widehat{Z}^{N,j}_n = \mathcal{NN}_n(t_n, X_n^{N,j}; \theta_n)$
            \If{$n = N-1$}
                \State $\widehat{Y}_{N}^{N,j} = g(X_{N}^{N,j})$
            \Else
                \State $\widetilde{Y}_{n+1}^{N,j} = \mathcal{NN}_{n+1}(t_{n+1}, X_{n+1}^{N,j}; \theta^{*,1}_{n+1})$
                \State $\widehat{Y}_{n+1}^{N,j} = \min(\max(\widetilde{Y}_{n+1}^{N,j}, -f_2(t_{n+1}, X_{n+1}^{N,j})), f_1(t_{n+1}, X_{n+1}^{N,j}))$ 
            \EndIf
            \EndFor
            \State $\ell(\theta_n) = \frac{1}{M} \sum_{j=1}^M |\widehat{Y}_{n+1}^{N,j} - (\widetilde{Y}_n^{N,j} - \varphi(t_n, X_n^{N,j})\Delta t + \widehat{Z}_n^{N,j} \Delta B^{N,j}_{n+1})|^2$
            \State Update parameters $\theta_n = \theta_n - r \nabla_{\theta_n} \ell(\theta_n)$
    \EndFor
    \State Save $(\theta^{*,1}_n, \theta^{*,2}_n) = (\theta^1_n, \theta^2_n)$
\EndFor\\
\Return $(\widehat{Y}^{N}_n, \widehat{Z}^N_n)$ for $n = 0, \ldots, N$
\end{algorithmic}
\end{algorithm}


\section{Convergence Analysis}
\label{sec:conv}
The main goal of this section is to prove the convergence of the neural network-based scheme towards the solution $(Y, Z)$ of the DRBSDE \eqref{eq:drbsde_intro} governing the Dynkin game and to show that the learned stopping times converge to the true optimal stopping times.

We first examine whether the algorithm correctly approximates the value function of the Dynkin game. The following theorem ensures that, as the time discretization $N$ increases and the neural network capacity grows, the solution converges to the true one.

\begin{theorem}
\label{theo:algo}
    Let $(\widehat{Y}_{n}^N,\widehat{Z}^N_n)$ for $n=0, \ldots, N$ be neural network approximations of the DRBSDE solution $(Y_t, Z_t)$ for $t \in [0,T]$. Then the error 
    \begin{align*}
        \max_{n=0, \ldots, N-1} \mathbb{E}\left[ |Y_{t_n} -  \widehat{Y}^N_n |^2\right] +  \sum_{n=0}^{N-1} \int_{t_n}^{t_{n+1}} \mathbb{E}\left[ |Z_{t} -  \widehat{Z}^N_n |^2\right] dt
        \end{align*}
        converges towards 0 as we increase number of timesteps $N$ and the number of neural networks' hidden parameters $\theta_n$ for each $n = 0, \ldots, N$.
\end{theorem}


\begin{proof}
Let us now present a sketch of the proof. Following the same approach as in \cite{bayraktar, pham}, the total approximation error can be decomposed into two main components: (i) an algorithmic error, representing the neural network approximation of the solution, and (ii) a discretization error, which results from the time discretization of the DRBSDE. \\

In the first two steps, we will establish convergence results for the algorithmic and time-discretization errors. Next, we will derive a bound for the error between the continuous process $Y_t$ and its time discretization $Y_n^N$. Then, we will derive an estimate for the difference between the discretized process $Y_n^N$ and its neural network approximation $\widehat{Y}_n^N$. By combining these bounds, we will obtain overall convergence for the total error. A similar reasoning will be applied to the process $Z_t$, yielding analogous bounds for its approximation.

\begin{enumerate}
    \item \textbf{Neural Network Approximation Error:} Define the neural network approximation errors as 
\begin{equation}
\begin{aligned}
   & \varepsilon_n^{\NN, y} \coloneqq \inf_{\theta_n^1} \bE \left[ |\hat{y}_n - \widetilde{Y}^N_n(X_n^N ; \theta_n^1)|^2 \right],\\ &\varepsilon_n^{\NN, z} \coloneqq \inf_{\theta_n^2} \bE \left[ |\hat{z}_n - \hat{Z}^N_n(X_n^N ; \theta_n^1)|^2 \right],
   \label{eq:nnerrors}
\end{aligned}
\end{equation}
where the auxiliary functions $\hat{y}_n,\hat{z}_n$ are defined as 
\begin{equation} 
\hat{y}_n \coloneqq \bE [ \hat{Y}^N_{n+1} | \mathcal{F}_{t_n}] + \varphi(t_n, X_n^N) \Delta t, 
\label{eq:y}
\end{equation}
\begin{equation} 
\hat{z}_n\coloneqq \frac{1}{ \Delta t}\bE [ \hat{Y}^N_{n+1} \Delta B_n| \mathcal{F}_{t_n}]. 
\label{eq:z}
\end{equation}
By the Universal Approximation Theorem, the errors in \eqref{eq:nnerrors} converge to zero as the number of parameters of the neural network go to infinity.\\

\item \textbf{Time Discretization Errors:} We denote by $Y_n^N$ the discrete approximation of $Y_t$ at time $t_n$. At the terminal time, this is given by \begin{equation*}
    Y_N^N= g(X_N^N).
\end{equation*}  For earlier time steps, the value $Y_n^N$ is computed recursively
\begin{equation}
    Y^N_n = \min(\max(\vardbtilde{Y}^N_{n}, - f_2(t_n, X_n^{N^j})),f_1(t_n, X^N_{n})), \qquad n = N-1, \ldots, 0.
    \label{eq:minmax2}
\end{equation}
Here $\vardbtilde{Y}_n^N$ follows the formulation stated in \cite{bouchard2004discrete}. Specifically, it holds  
\begin{equation} 
\vardbtilde{Y}_n^N \coloneqq \bE [ Y_{n+1}^N | \mathcal{F}_{t_n}] + \varphi(t_n, X_n^N) \Delta t,
\label{eq:Y}
\end{equation}
\begin{equation} 
Z_n^N \coloneqq \frac{1}{ \Delta t}\bE [Y_{n+1}^N \Delta B_n| \mathcal{F}_{t_n}]. 
\label{eq:Z}
\end{equation}
By using established results in the literature (e.g. \cite{chassagneux2009discrete, pham}), under Assumption \ref{ass1} and  Assumption \ref{ass 2}, the following convergenge results hold

\begin{equation}
    \max_{n \in \{0,1,\dots,N-1\}} \bE \left[ |\widetilde{Y}_{t_n} - \vardbtilde{Y}^N_{n}|^2 + |Y_{t_n} - Y^N_{n}|^2 \right] \to 0 \quad \text{as } N \to \infty,
    \label{eq:err1}
\end{equation}
\begin{equation}
    \bE \left [\sum_{n=0}^{N-1} \int_{t_n}^{t_{n+1}} | Z_t - Z^N_{n}|^2 \right] dt \to 0 \quad \text{as } N \to \infty. 
    \label{eq:err2}
\end{equation}
\item \textbf{Error Bound for $\mathbf{Y_t}$:} 
By \eqref{eq:minmax} and \eqref{eq:minmax2}, we observe
\begin{equation}
   \bE \left [ | Y_{n}^N - \widehat{Y}^N_n |^2 \right]\leq  \bE \left[| \vardbtilde{Y}^N_{n} - \widetilde{Y}^N_n |^2 \right].
   \label{eq:Yt_bound}
\end{equation}
By using \eqref{eq:y}, \eqref{eq:Y} and Young's inequality, we derive the following upper bound for sufficiently small $\Delta t$:
\begin{equation}
   \bE \left[| \vardbtilde{Y}^N_{n} - \widetilde{Y}^N_n |^2 \right]  \le (1 + C \Delta t) \bE \left [ | Y_{n+1}^N - \widehat{Y}^N_{n+1} |^2 \right] + \frac{C}{\Delta t} \bE \left[| \widetilde{Y}^N_{n} - \hat{y}_{n} |^2 \right]. 
   \label{eq:continuation_bound}
\end{equation}
From now on, we denote by $C$ a positive constant which is independent of the neural network structure and may vary from line to line, depending on the specific estimate or bound being used.\\

One would perhaps feel inclined to bound, under the optimal choice of trainable parameters $\theta^1_n$, the last term in inequality \eqref{eq:continuation_bound} by $\varepsilon^{\NN,y}_n$, but we need to be careful. Recall that the optimal parameters $\theta_n^{*,1}$ obtained in the Algorithm \ref{algo:DRBSDE} do not, in general, coincide with the ones required to achieve $\varepsilon^{\NN,y}_n$. However, through analysis of the loss function $\ell(\theta_n)$ in the Algorithm \ref{algo:DRBSDE}, it is possible to obtain an estimation

\begin{equation*}
\begin{aligned}
       \bE \left[ | \widetilde{Y}^N_n- \hat{y}_n|^2 \right]\le \varepsilon^{\NN,y}_n + \Delta t \varepsilon^{\NN,z}_n.
\end{aligned}
\end{equation*}
By combining this with \eqref{eq:Yt_bound} and \eqref{eq:continuation_bound}
\begin{equation*}
  \bE \left [ | Y_{n}^N - \widehat{Y}^N_n |^2 \right]  \le (1 + C \Delta t) \bE \left [ | Y_{n+1}^N - \widehat{Y}^N_{n+1} |^2 \right] + C( \frac{1}{\Delta t}  \varepsilon^{\NN,y}_n + \varepsilon^{\NN,z}_n). 
\end{equation*}
By induction on the right hand side and incorporating the convergence result \eqref{eq:err1}, we established an error bound for $Y_t$, which ensures that the error tends to zero as the number of timesteps and network parameters increase.

\item \textbf{Error bound for $\mathbf{Z_t}$}:  One should also verify that the $Z$ component is well approximated. Using similar reasoning as we did for the $Y$ component, we get an estimation
$$
 \Delta t \bE \left[ |  \widehat{Z}_n^{N} - \hat{z}_n|^2 \right] \leq \varepsilon^{\NN,y}_n + \Delta t \varepsilon^{\NN,z}_n.
$$

By using triangular inequality it follows
\begin{equation*}
    \Delta t \bE \left [ | Z_n^N - \widehat{Z}_n^N |^2 \right ] \le 2 \Delta t \bE \left [ | Z_n^N - \widehat{z}_n^N |^2 \right ] + C(\varepsilon^{\NN,y}_n + \Delta t \varepsilon^{\NN,z}_n). 
\end{equation*}
Now by \eqref{eq:z}, \eqref{eq:Z} and Cauchy-Swartz inequality, we get
\begin{equation*}
  \sum_{n=0}^{N-1} \Delta t \bE \left [ | Z_n^N - \widehat{z}_n^N |^2 \right ] \le d\sum_{n=0}^{N-1}  \bE \left [ | Y_{n+1}^N - \widehat{Y}_{n+1}^N |^2 \right ] \le C \sum_{n=0}^{N-1}( \frac{1}{\Delta t}  \varepsilon^{\NN,y}_n + \varepsilon^{\NN,z}_n).
\end{equation*}
Similarly to what is done above for $Y$, the error between the continuous process $Z_t$ and its neural network approximation is made up of the above quantity and the discrete approximation error coming from \eqref{eq:err2}. 
\end{enumerate}

\end{proof}

We now verify that the stopping times generated by the algorithm converge to the true optimal stopping times. The following theorem establishes that as $N$ increases, they converge to the saddle point of the game on $[0,T]$, aligning with their theoretical counterparts. \\

To prove the theorem, we first present the following lemma.

\begin{lemma}\label{lemma:cont}
    For $a \in \bR$, define the map $F : C([0,T]) \to [0,T]$ given by 
    $$
    F(g) := \inf\{t \geq 0 | g(t) \geq a \} \wedge T,
    $$
    where  \( C([0,T]) \) is the space of continuous functions on \([0,T]\).
    If $g \in C([0,T])$ does not have a local maximum at $F(g)$, then $F$ is continuous at $g$.
\end{lemma}

\begin{proof}
    Suppose $g$ does not have a local maximum at $F(g).$ For each $ \varepsilon >0$, there exists $t_0 \in (F(g), F(g) + \varepsilon)$ such that $g(t_0) > a$. Let $\lVert \widetilde{g} -g \rVert_{\infty} < g(t_0)-a$. Then, we have $\widetilde{g}(t_0) \geq a$ which yields $F(\widetilde{g}) \leq t_0$. Consequently, $F(\widetilde{g}) - F(g) \leq t_0 - F(g) < \varepsilon$. 

    Now, set $t_1 := F(g) - \varepsilon$ and $\bar{g} := \sup_{t \in [0,t_1]} g(t)$. By the definition of $F(g)$, we have $\bar{g}<a$. If $\lVert \widetilde{g} -g \rVert_{\infty} < a-\bar{g}$, then $F(\widetilde{g}) \geq t_1$, which implies $F(g) - F(\widetilde{g}) \leq F(g) - t_1 = \varepsilon$. Hence, $F$ is continuous in $g$.  
\end{proof}

\begin{theorem}\label{thm:time-conv}
    Assume $f_1,f_2 \in C^{1,2}([0,T]\times \mathbb{R}^d)$ and define 
    \begin{align*}
    &\widehat{\tau}^N_2 := \inf \{t_n \geq 0 : \widehat{Y}_{n}^N \leq -f_2(t_n, X_n^{N}) \} \wedge T, \\
    &\widehat{\tau}^N_1 := \inf \{t_n \geq 0 : \widehat{Y}_{n}^N \geq f_1(t_n, X_n^{N}) \} \wedge T.
    \end{align*}
    Then, $ \mathbb{E} \left[ |\widehat{\tau}^N_2 - \tau^*_{2,0}|^2 \right]$ and $\mathbb{E}  \left[ |\widehat{\tau}^N_1 - \tau^*_{1,0}|^2 \right]$ converge to $0$ as the number of timesteps $N$ and the number of neural networks' hidden parameters $\theta_n$ (for $n = 0, \ldots, N$) increase.
\end{theorem}

\begin{proof}
    We analyze the convergence of $\widehat{\tau}^N_1$, with the result for $\widehat{\tau}^N_2$ following analogously. 

    Let us start by denoting the continuous linear extension of the approximation $X_{n}^N$ for $n=0,\ldots, N$ as 
    $$
    X_t^{C,N} = X_{n}^N + \frac{t-t_n}{t_{n+1}-t_n}(X_{n+1}^N - X_{n}^N), \quad t \in [t_n,t_{n+1}],
    $$
    for $n = 0, \ldots, N-1.$ It has been shown in \cite{em-orig} that the sequence of continuous processes $X^{C,N}$ converges to $X$ in $\mathbb{L}^2(C[0,T])$.
    
    Now, recall the piecewise continuous process $\widetilde{Y}$ defined in Equation \eqref{eq:bem}, and set $$\widetilde{\tau}_1 := \inf \{t \geq 0 : \widetilde{Y}_{t} \geq f_1(t, X_t) \} \wedge T.$$

    By construction,
    \begin{align}\label{eq:t-tt}
        | \tau^*_{1,0} - \widetilde{\tau}_1| \leq \Delta t,
    \end{align}
    almost surely. 
    
   Similarly, define the continuous linear extension of $\widetilde{Y}_{n}^N$ for $n=0,\ldots, N$ as 
    $$
    \widetilde{Y}_t^{C,N} = \widetilde{Y}_{n}^N + \frac{t-t_n}{t_{n+1}-t_n}(\widetilde{Y}_{n+1}^N - \widetilde{Y}_{n}^N), \quad t \in [t_n,t_{n+1}],
    $$
    for $n = 0, \ldots, N-1.$ Furthermore, we put 
    \begin{align*}
    \widetilde{\tau}^{C,N}_1 := \inf \{t \geq 0 : \widetilde{Y}_t^{C,N} \geq f_1(t,X^{C,N}_t) \} \wedge T.
    \end{align*} 
    Again, by construction 
    \begin{align}\label{eq:th-tt}
        | \widetilde{\tau}^{C,N}_1 - \widehat{\tau}^{N}_1| \leq \Delta t,
    \end{align}
    almost surely. 
    
    Due to growth conditions in Assumption \ref{ass 2}, we have that 
    \begin{equation}\label{eq:uni-tilde}
        \max_{n=0,\ldots,N-1} \mathbb{E}\left[ \sup_{t \in [t_n, t_{n+1})} |\widetilde{Y}_t - \widetilde{Y}_{t_n}|^2 \right] \to 0 \text{ as } N \to \infty.
    \end{equation}

    This leads to 
    \begin{align*}
        \max_{n=0,\ldots,N-1} &\bE\left[ \sup_{t \in [t_n, t_{n+1})} |\widetilde{Y}_t - \widetilde{Y}^{C,N}_t|^2 \right] \\
         \leq \max_{n=0,\ldots,N-1} & \bE\left[ \sup_{t \in [t_n, t_{n+1})} \left\{|\widetilde{Y}_t - \widetilde{Y}^{N}_n|^2 + |\widetilde{Y}_t - \widetilde{Y}^{N}_{n+1}|^2 \right\} \right] \\
        \leq  \max_{n=0,\ldots,N-1} &\bE\left[ \sup_{t \in [t_n, t_{n+1})} \left\{|\widetilde{Y}_t - \widetilde{Y}^{N}_n|^2 + |\widetilde{Y}_t - \widetilde{Y}_{t_{n+1}} |^2 + |\widetilde{Y}_{t_{n+1}} - \widetilde{Y}^{N}_{n+1}|^2 \right\} \right],
    \end{align*}
    which converges to 0 as $N \to \infty$ due to \eqref{eq:err1} and \eqref{eq:uni-tilde}. Since $\widetilde{Y}$ is continuous on $[t_n, t_{n+1})$ for each $n=0, \ldots, N-1$, we can paraphrase the above in the following way: for each $n$, the random variable $\widetilde{Y}^{C,N}:\Omega \to C([t_n, t_{n+1}))$ converges in $\mathbb{L}^2$ to $\widetilde{Y} : \Omega \to C([t_n, t_{n+1}))$ as $N \to \infty$. 
    
    Let us define the bounded maps $F_n : C([t_n, t_{n+1})) \to [t_n,t_{n+1}] \cup \{ \infty \}$ by
    $$
    F_n(g_n) = \inf\{t \in [t_n, t_{n+1})\; |\; g_n(t) \geq 0 \}, \quad n = 0, \ldots, N-1,\text{}
    $$
    and set $F : C([0,T]) \to [0,T]$ as $F(g) := \min \{ F_0(g), \ldots , F_{N-1}(g),T \},$ where $g$ is a piecewise continuous function 
    $$g(t) = \sum_{n=0}^{N-1}g_n(t) \mathbbm{1}_{[t_n, t_{n+1})}(t).$$ 

    It is clear that $F(\widetilde{Y}_{\cdot} - f_1(\cdot, X_{\cdot})) = \widetilde{\tau}_1$ and $F(\widetilde{Y}_{\cdot}^{C,N} - f_1(\cdot,X_{\cdot}^{C,N})) = \widetilde{\tau}_1^{C,N}.$ Convergence of $\widetilde{\tau}_1^{C,N}$ in $\mathbb{L}^2$ towards $\widetilde{\tau}_1$ is ensured by the continuous mapping theorem if $F$ is continuous at $\widetilde{Y}_{\cdot} - f_1(\cdot, X_{\cdot})$ almost surely. Since $\mathbb{P}(\Tilde{\tau_1}\in \pi)=0$ it suffices to show continuity of $F_n$ for each $n = 0, \ldots, N-1.$ 

  If $\widetilde{Y}_{t} - f_1(t, X_{t})$ is strictly positive or negative on $[t_n, t_{n+1})$, continuity of $F_n$ is clear. Let us now investigate a non-trivial case where the above function passes through zero, i.e., $F_n(\widetilde{Y}_{t} - f_1(t, X_{t})) \in [t_n, t_{n+1})$. 

    First, note that $\widetilde{Y}$ can be written as a forward SDE on $[t_n, t_{n+1})$: 
    $$
    \widetilde{Y}_t = \widetilde{Y}_{t_n} -  \displaystyle\int_{t_n}^{t} \varphi(s, X_s) ds
   + \displaystyle\int_{t_n}^{t} Z_s dB_s.
    $$
    Since $f_1 \in C^{1,2}([0,T]\times \mathbb{R}^d)$, applying It\^{o} formula allows us to express $\widetilde{Y}_{t} - f_1(t, X_{t})$ as an SDE as well. It is a well-known fact that the probability of Brownian motion having a local maximum at finite stopping time $\tau$ equals $0$. The same claim can be extended to SDEs due to the Girsanov theorem, which means that $\widetilde{Y}_{\cdot} - f_1(\cdot, X_{\cdot})$ does not have a local maximum at $F_n(\widetilde{Y}_{\cdot} - f_1(\cdot, X_{\cdot}))$ almost surely. By Lemma \ref{lemma:cont}, $F_n$ is continuous at $Y$ almost surely for each $n$, which means that $\bE \left[ |\widetilde{\tau}^{C, N}_1 - \widetilde{\tau}^1|^2 \right] \to 0$ as $N \to \infty.$ Due to estimations in \eqref{eq:t-tt} and \eqref{eq:th-tt} we get that 
    $\widehat{\tau}^{ N}_1$ converges to $\tau^*_{1,0}$ in $\mathbb{L}^2$ which concludes the proof. 
\end{proof}

We have the following corollary.

\begin{corollary}
    Using notations and assumptions of Theorem \ref{thm:time-conv}, we have that for each $x \in \mathbb{R}^d$
    $$
    |V(0,x) - J_{0,x}(\widehat{\tau}^{ N}_1, \widehat{\tau}^{ N}_2) |
    $$
    converges towards 0 as we increase the number of timesteps $N$ and the number of neural networks' hidden parameters $\theta_n$ for each $n = 0, \ldots, N$.
    \label{cor:value}
\end{corollary}

\begin{proof}
    By Theorem \ref{thm:time-conv}, the estimated exit times $(\widehat{\tau}^{ N}_1, \widehat{\tau}^{ N}_2)$ converge to the optimal stopping times $(\tau_{1,0}^*, \tau_{2,0}^*)$ in $\mathbb{L}^2$, which implies weak convergence. Since the functional $J_{0,x}(\tau_1,\tau_2)$, defined in \eqref{dynkin_game}, is a continuous map of exit times, the result follows immediately. 
\end{proof}

\section{Numerical Implementation}
\label{sec:num experiments}
To solve the DRBSDE \eqref{DRBSDEdecay}, we implemented a feedforward neural network using PyTorch. The network follows the backward-in-time algorithm presented in Algorithm \ref{algo:DRBSDE}, designed to approximate the solution pair $(Y,Z)$. It consists of $L=3$ hidden layers, each with $N_\ell = 50$ neurons for $\ell = 1, 2, 3$. The input layer has $N_0 = d + 1$ neurons representing time and the forward process, while the output layer has $N_4=  d + 1$ neurons corresponding to the estimated values of $Y$ and $Z$. We use \emph{tanh} as activation function in all hidden layers and perform the optimization using Adam algorithm, with a learning rate $lr = 0.001$.  \\

We train the network for a total of $N=50$ time steps, we use $500$ training epochs for the first two optimization steps (corresponding to the last two time steps in the backward-in-time algorithm) and $100$ epochs for the remaining steps. The batch size is set to $M= 2^{10}$, and the training set is generated as described in Subsection \ref{discretization}. The input data is standardized before being passed to the neural network.  \\

In the following, we report the implementation of two different problems. The first one is designed to be a fair game, constructed symmetrically to serve as a benchmark solution. This allows us to evaluate the performance of the algorithm in a controlled setting. The second is designed to have a greater economic interest by illustrating how the CfDs introduced in Section \eqref{sec:application} actually works in energy markets.  

\subsection{Benchmark Problem: Symmetric Fair Game}
We first validate our numerical solver using a multi-dimensional symmetric Dynkin game with constant barriers. This serves as a benchmark for testing the DRBSDE architecture and training algorithm.\\

The forward process governing the state dynamics is modeled as a $d-$dimensional OU process: \begin{equation*}
    dX_s = - \kappa  X_s ds + \sigma dB_s,
\end{equation*}
where $B_s$ is a $d$-dimensional Brownian Motion and $\kappa, \sigma \in \bR^{d\times d}$.  The process starts at $X_0 = 0$.\\
We now define the reward structure of the Dynkin game introduced in \eqref{dynkin_game}. The payoff consists of a running cost/reward accumulated up to the first stopping time, along with an exit value  that depends on which player stops the game: 
\begin{align*} \label{benchmark}
J_{0,0}(\tau_1, \tau_2) = \mathbb{E}\bigg[
-\alpha  \int_0^{\tau_1 \wedge \tau_2}  \frac{1}{d}\mathbf{1}^\top X_s  ds
&+ \gamma \mathbbm{1}_{\{\tau_1 \leq \tau_2, \tau_1 < T\}}
- \gamma \mathbbm{1}_{\{\tau_2 < \tau_1\}} 
\bigg], \quad \tau_1, \tau_2 \in \mathcal{T}_{0,T}
\end{align*}
with $\alpha > 0 $ scaling parameter. The exit payments are taken to be constants and symmetric with $\gamma > 0$. The integrand function denotes the average of the components of \(X_s \) and $\mathbf{1} \in \bR^d$ is the vector of ones. \\
The value process \( Y  \) of the Dynkin game is characterized as the unique solution to the following DRBSDE, consistent with the general formulation introduced in \eqref{eq:drbsde_intro} and its accompanying barrier and minimality conditions:
\begin{equation*}
    \begin{cases}
        Y_s  = \displaystyle - \frac{\alpha}{d} \int_s^T  \mathbf{1}^\top X_r \, dr 
        - \int_s^T Z_r \, dB_r 
        + A_T - A_s - (C_T - C_s), \\[10pt]
        -\gamma \leq Y_s \leq \gamma, \quad \forall s \in [0,T], \\[10pt]
        \displaystyle \int_0^T (Y_r + \gamma) \, dA_r = 0, 
        \quad \int_0^T (\gamma - Y_r) \, dC_r = 0.
    \end{cases}
\end{equation*}

Due to the symmetric nature of the problem, both players are in identical strategic positions. This implies that neither player has any structural advantage and therefore the game is fair. The value of the game at time $t=0$ is $Y_0 = 0$. Furthermore, since the game dynamics and exit rules are symmetric, the exit times $\tau_1, \tau_2$ must follow the same probability distribution.  \\
\begin{table}[h] 
\centering
\begin{tabular}{|l|c|}
\hline
\textbf{Simulation parameters} & \textbf{Values} \\
\hline
\( \text{Dimension } d \)                     & 20 \\
\( \text{Time horizon } T \)                  & 1.0 \\
\( \text{Mean reversion matrix } \kappa \)    & diag\((\kappa_1, \ldots, \kappa_{d}),\ \kappa_i \sim \text{Uniform}[1.5, 2.5]\) \\
\( \text{Volatility matrix } \sigma \)        & diag\((1.0, \ldots, 1.0)\) \\
\( \text{Initial value } x \)                 & 0.0 \\
\( \text{Scaling factor } \alpha \)           & 10.0 \\
\( \text{Barrier coefficient } \gamma \)      & 0.5 \\
\hline
\end{tabular}
\caption{Parameter values for the high-dimensional benchmark problem}
\label{table:model_parameters_1}
\end{table}

Since the algorithm is local, we have separate loss functions for all the time steps. In Figure \ref{fig:loss}, we can see how the loss functions for $n \ge N-3$ decrease towards zero. The losses for $n < N-3$ behave in a similar way, so we omit their presentation. \\In Figure \ref{fig:Yt}, we see three realizations of the $Y_t$ dynamics, while in Figure \ref{fig:Y0} the distribution of the estimated value $Y_0$ is presented. As we can see, the algorithm is able to capture the fair nature of the game.

\begin{figure}[H]
    \centering
    \begin{subfigure}{0.32\textwidth}
        \includegraphics[width=\linewidth]{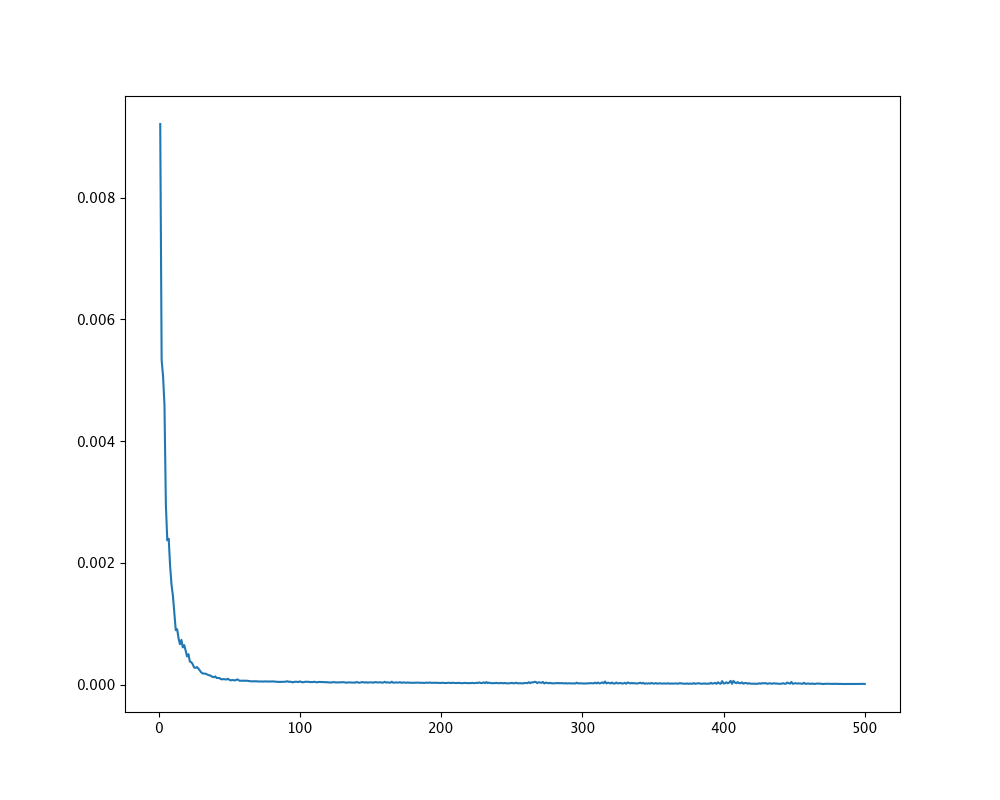}
        \caption{$n =  N-1$}
    \end{subfigure}
    \hfill
    \begin{subfigure}{0.32\textwidth}
        \includegraphics[width=\linewidth]{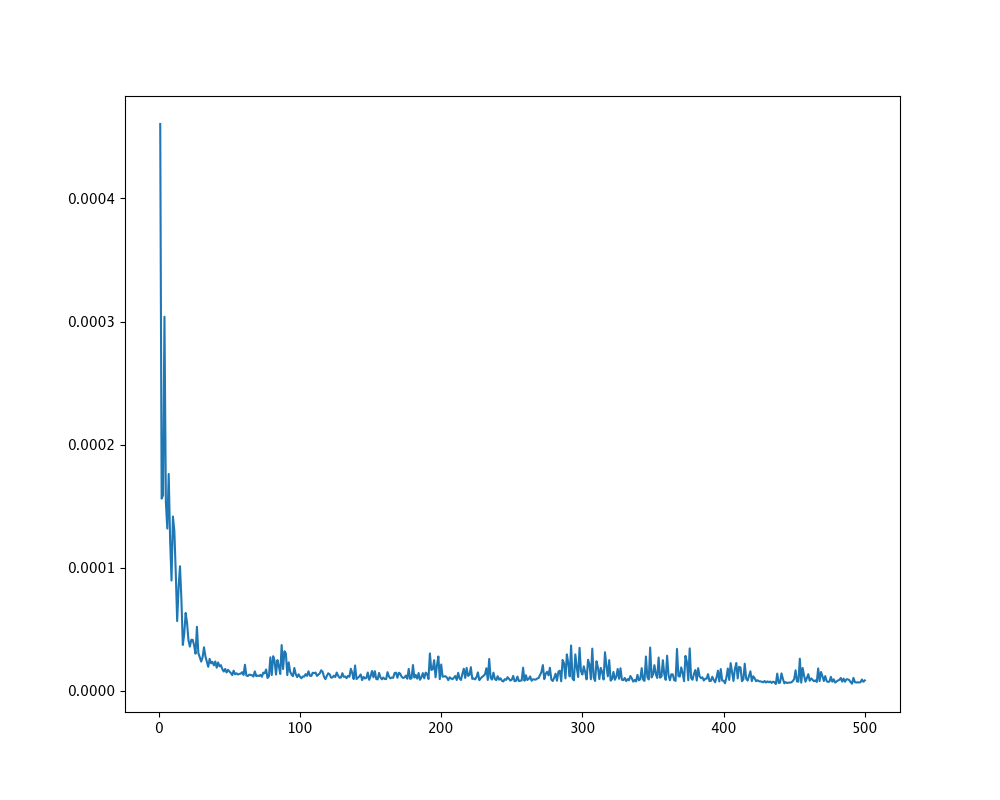}
        \caption{$n =  N-2$}
    \end{subfigure}
    \hfill
    \begin{subfigure}{0.32\textwidth}
        \includegraphics[width=\linewidth]{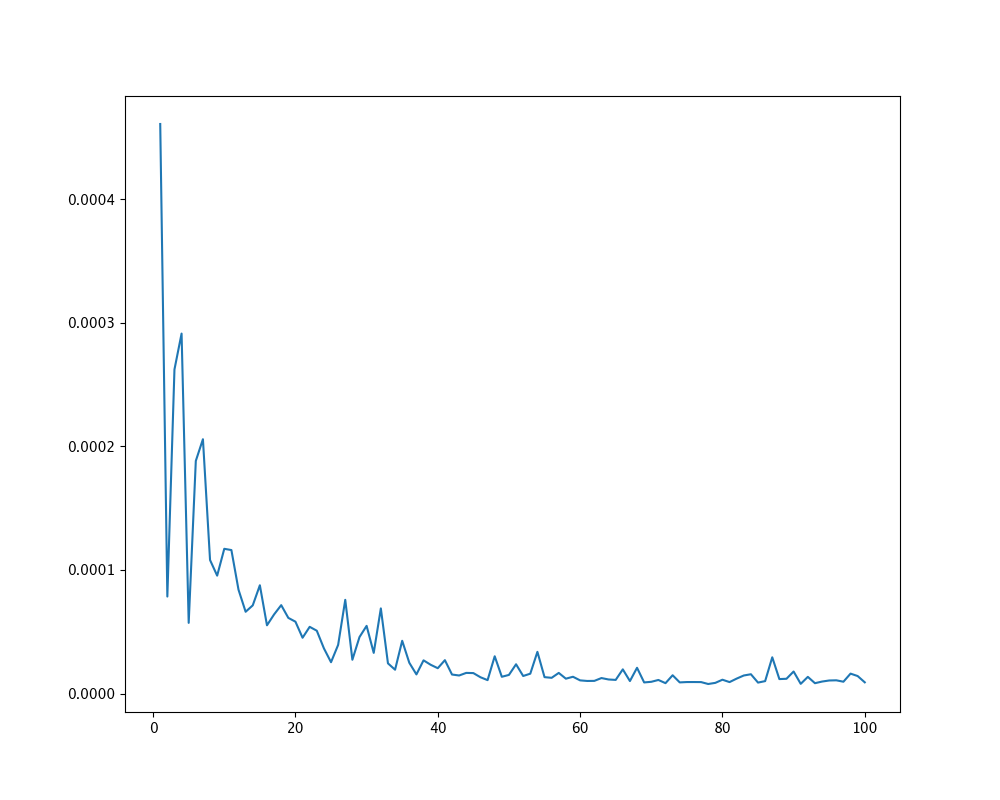}
        \caption{$n =  N-3$}
    \end{subfigure}
    \caption{Convergence of loss function at different timesteps $n$ in the benchmark problem}
    \label{fig:loss}
\end{figure}

\begin{figure}[H]
    \centering
    \begin{subfigure}[t]{0.49\textwidth}
        \includegraphics[width=\linewidth]{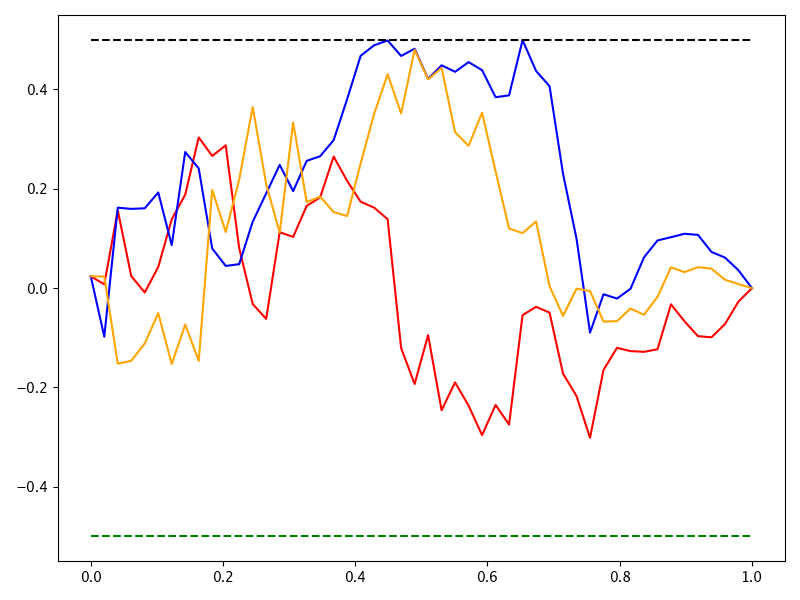}
        \caption{Realizations of $Y_t$}
        \label{fig:Yt}
    \end{subfigure}
    \hfill
    \begin{subfigure}[t]{0.49\textwidth}
        \includegraphics[width=\linewidth]{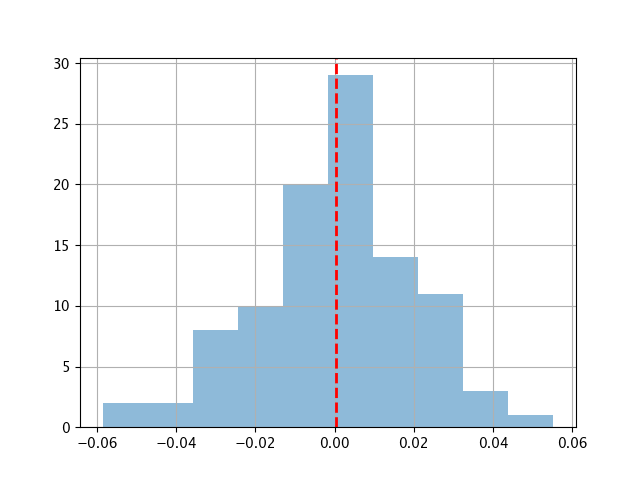}
        \caption{Distribution of estimated $Y_0$ compared to true value $Y_0 = 0$}
        \label{fig:Y0}
    \end{subfigure}
    \caption{Three realisations of the $Y_t$ dynamics and the distribution of estimated $Y_0$ over 100 independent training processes in the benchmark problem}
    \label{fig:Y}
\end{figure}

Another topic of interest is the distribution of the first exit times $\tau_1$ and $\tau_2$. As shown in Figure~\ref{fig:times}, these appear to follow the same distribution, as expected from the symmetric nature of the benchmark game. Furthermore, we can observe that over time the probability of exiting the game decreases due to the dynamics being drawn to the terminal value $Y_T = 0.$ In approximately 85 \% of realisations the barriers are not reached. Conditionally on the event that the player exits the game, the expected exit time is approximately 0.31. 

\begin{figure}[H]
    \centering
    \includegraphics[width=0.5\linewidth]{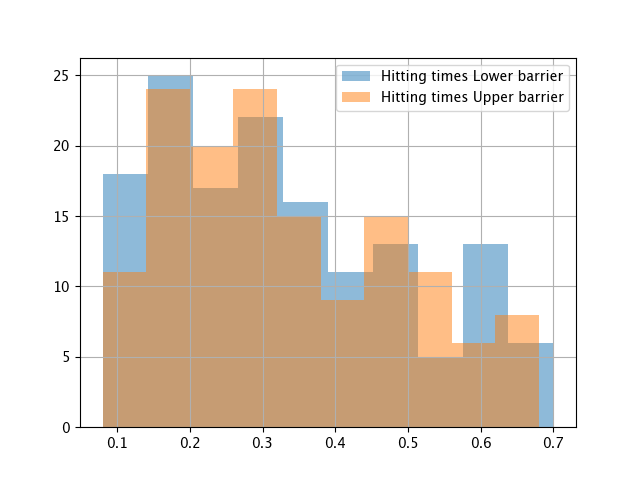}
    \caption{Estimated distribution of the first exit times in the benchmark problem}
    \label{fig:times}
\end{figure}

\subsection{Parameter Calibration for the Underlying Price Process} \label{calibration} 
In this section, we present the calibration of the parameters $\mu, \sigma$ and $\kappa$ governing the OU process \eqref{eq:OU} 
for dimension $d=24$.
The calibration is based on historical data of the weekly spot power prices for the countries of continental Europe\footnote{Historical data on European spot electricity prices is available at \url{https://ember-energy.org/data/european-wholesale-electricity-price-data/}}, the list of which is available in Table \ref{tab:eu-params}. We use the data over a two-year interval from July 2023 to July 2025. The reason for choosing weakly prices for the calibration is that they follow a more regular pattern, which makes them more suitable when working with models without jumps. \\

To estimate the parameters, we perform maximum likelihood estimation (MLE) separately for each country. Given the discrete observations of the $j$-th price process \( X^j_t \) at times \( t_0, t_1, \dots, t_N \)  for $j = 1, \ldots, d$, this method exploits the fact that the conditional distribution of \( X^j_{t_{i+1}} \) given \( X^j_{t_i} \) for an OU process is  

\begin{equation*}
X^j_{t_{i+1}} \mid X^j_{t_i} \sim \mathcal{N} \left( X^j_{t_i} + \kappa^j(\mu^j - X^j_{t_i})\Delta t, \,(\sigma^j)^2\Delta t \right),
\end{equation*}
where $\Delta t$ is the time step. Note that using the notation of Section \ref{sec:market-price}, the $\mu^j$ represent the components of vector $\mu \in \mathbb{R}^d$, while $\kappa^j$ and $\sigma^j$ are the diagonal elements of matrices $\kappa, \sigma \in \mathbb{R}^{d \times d}$, respectively.

To estimate \( \mu^j, \sigma^j, \) and \( \kappa^j \), we maximize the log-likelihood function  

\begin{equation}
\log \mathcal{L}(\mu^j, \sigma^j, \kappa^j) = \sum_{i=0}^{N-1}- \frac{1}{2} \left[  \log (2 \pi (\sigma^j)^2\Delta t) + \frac{(X^j_{t_{i+1}} - (X^j_{t_i} + \kappa^j(\mu^j - X^j_{t_i})\Delta t) )^2}{2(\sigma^j)^2\Delta t} \right].
\end{equation}

To perform the optimization we used the L-BFGS-B algorithm, a quasi-Newton optimization method which approximates the inverse Hessian matrix without storing it entirely and thus significantly reducing memory requirements.\\

\begin{table}[htbp]
\centering
\begin{tabular}{|l||r|r|r|r|}
\hline
        \textbf{Country} &  \textbf{diag}$(\bm{\kappa})$ &  $\bm{\mu}$ &  \textbf{diag}(\bm{$\sigma$})&  \textbf{p-value} \\
\hline
\hline
        Austria &     27.43 &  90.69 &    187.93 &     0.44 \\
\hline
        Belgium &     41.72 &  80.56 &    210.00 &     0.56 \\
\hline
       Bulgaria &     27.25 & 104.95 &    222.49 &     0.62 \\
\hline
        Croatia &     31.81 & 100.26 &    209.58 &     0.93 \\
\hline
        Czechia &     30.60 &  91.69 &    188.54 &     0.75 \\
\hline
        Denmark &     73.96 &  75.68 &    240.96 &     0.88 \\
\hline
        Estonia &     75.98 &  88.16 &    325.86 &     0.17 \\
\hline
         France &     30.44 &  69.10 &    225.84 &     0.82 \\
\hline
        Germany &     53.40 &  84.53 &    220.38 &     0.29 \\
\hline
         Greece &     24.90 & 104.93 &    177.39 &     0.55 \\
\hline
        Hungary &     30.72 & 104.50 &    243.72 &     0.43 \\
\hline
          Italy &     18.47 & 114.84 &    110.14 &     0.82 \\
\hline
         Latvia &     71.99 &  89.50 &    302.16 &     0.32 \\
\hline
      Lithuania &     72.55 &  89.27 &    306.12 &     0.33 \\
\hline
     Luxembourg &     53.40 &  84.53 &    220.38 &     0.29 \\
\hline
     Montenegro &     22.40 & 105.45 &    157.00 &     0.83 \\
\hline
    Netherlands &     51.74 &  84.75 &    204.55 &     0.19 \\
\hline
North Macedonia &     24.95 & 106.78 &    192.44 &     0.26 \\
\hline
         Poland &     52.28 &  99.17 &    181.74 &     0.98 \\
\hline
       Portugal &     20.46 &  70.49 &    204.60 &     0.87 \\
\hline
        Romania &     27.38 & 105.87 &    229.18 &     0.76 \\
\hline
         Serbia &     24.56 & 104.25 &    198.08 &     0.63 \\
\hline
       Slovakia &     27.29 &  98.80 &    204.53 &     0.78 \\
\hline
       Slovenia &     32.92 &  98.28 &    205.76 &     0.72 \\
\hline
          Spain &     20.95 &  71.03 &    209.19 &     0.95 \\
\hline
    Switzerland &     17.22 &  91.26 &    164.64 &     0.31 \\
\hline
\end{tabular}

\caption{Calibrated OU parameters for continental European electricity markets}
\label{tab:eu-params}
\end{table}

To evaluate the model fit, we analyse the residuals, that is, the deviations between the observed electricity prices and the drift component estimated from the calibrated OU process. A Kolmogorov Smirnov (K-S) goodness of fit test is performed for each country, with resulting $p$-values reported in Table~\ref{tab:eu-params}. All values are well above the standard $0.05$ significance level, indicating no statistical evidence against the chosen model specification. Estonia yields the lowest $p$-value at $0.17$, while Poland displays the strongest statistical alignment.\\
To complement these tests, Figure~\ref{fig:main} provides diagnostic plots for the two representative cases Poland and Estonia. Figures \ref{fig:comparison_prices_Poland} and \ref{fig:comparison_prices_Estonia} display the observed electricity price series for Poland and Estonia, alongside simulated price paths generated from the calibrated OU process. 
The estimated price path shown in the figure is a single realisation of the calibrated OU process, and as such, it is inherently random. It is not meant to align pointwise with the historical price trajectory, but rather to reflect the statistical behavior implied by the calibrated dynamics, such as mean reversion and volatility.
Additionally, Subplots~\ref{fig:sub1} and~\ref{fig:sub4} show that the empirical and simulated price distributions align closely. Subplots~\ref{fig:sub2} and~\ref{fig:sub5} depict the autocorrelation function (ACF) of the residuals, while the $Q-Q$ plots in Subplots~\ref{fig:sub3} and~\ref{fig:sub6} further confirm a good distributional fit.\\

\begin{figure}[htbp]
    \centering

    \begin{subfigure}[t]{0.4\textwidth}
        \includegraphics[width=\textwidth]{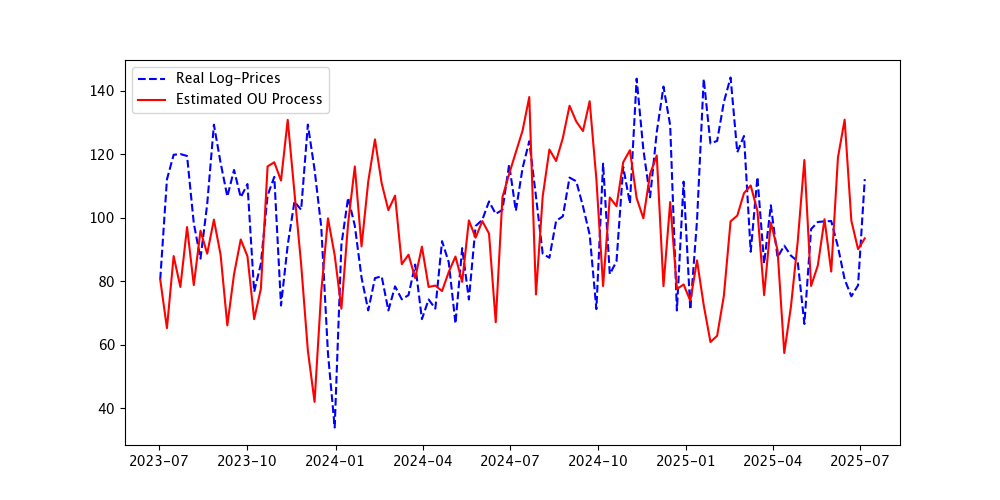}
        \caption{Poland - observed (blue dashed line) and simulated prices (red solid line).}
        \label{fig:comparison_prices_Poland}
    \end{subfigure}
    \hspace{0.05\textwidth}
    \begin{subfigure}[t]{0.4\textwidth}
        \includegraphics[width=\textwidth]{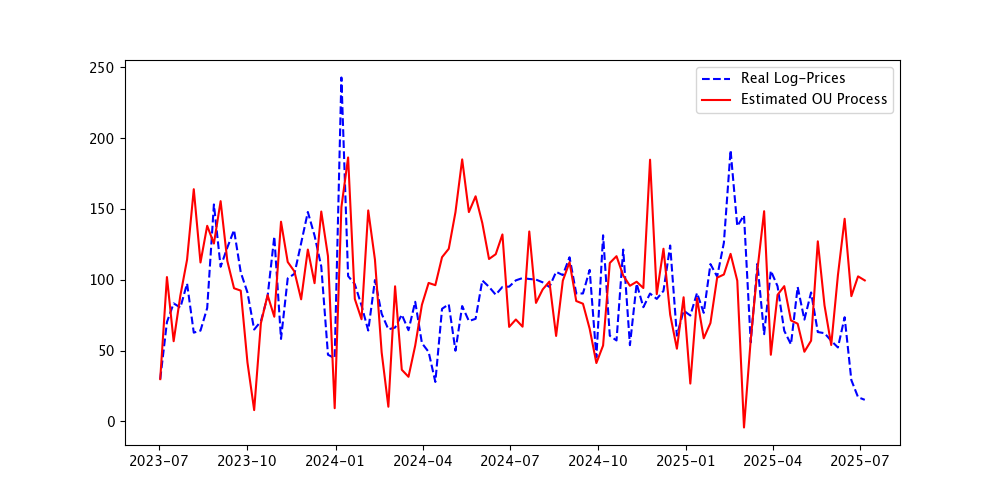}
        \caption{Estonia - observed (blue dashed line) and simulated prices (red solid line)}
        \label{fig:comparison_prices_Estonia}
    \end{subfigure}

    \begin{subfigure}[t]{0.3\textwidth}
        \includegraphics[width=\textwidth]{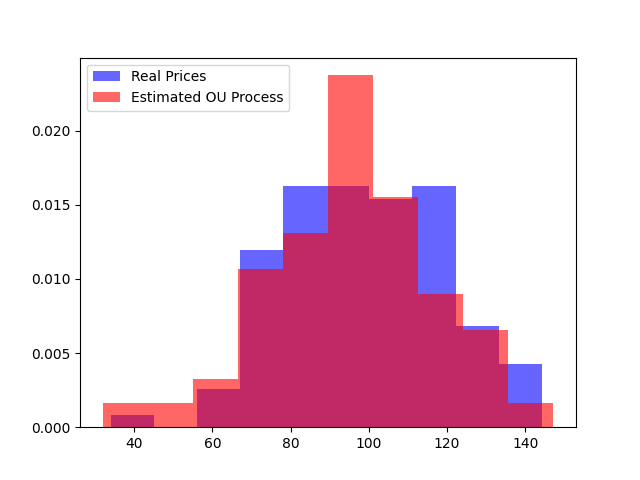}
        \caption{Poland - distribution comparison}
        \label{fig:sub4}
    \end{subfigure}
    \hspace{0.05\textwidth}
        \begin{subfigure}[t]{0.3\textwidth}
        \includegraphics[width=\textwidth]{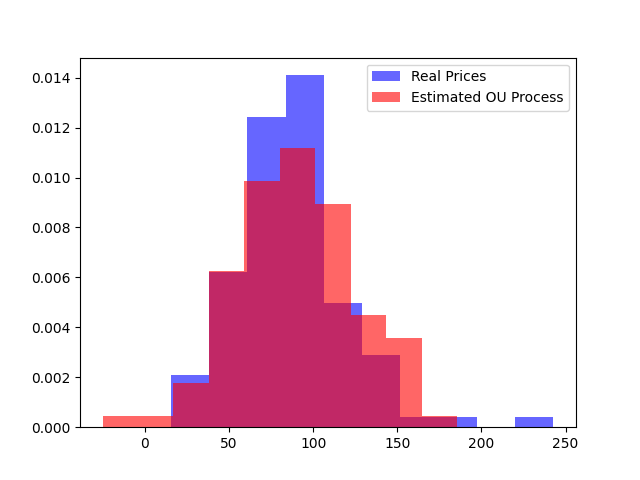}
        \caption{Estonia - distribution comparison}
        \label{fig:sub1}
    \end{subfigure}


    \begin{subfigure}[t]{0.3\textwidth}
        \includegraphics[width=\textwidth]{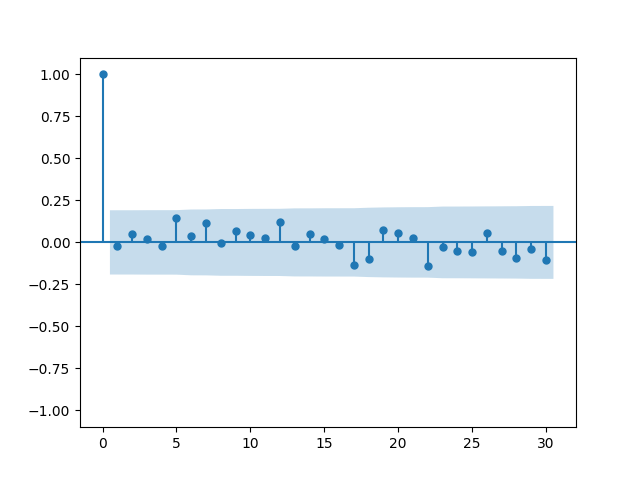}
        \caption{Poland - residuals ACF}
        \label{fig:sub5}
    \end{subfigure}
\hspace{0.05\textwidth}
       \begin{subfigure}[t]{0.3\textwidth}
        \includegraphics[width=\textwidth]{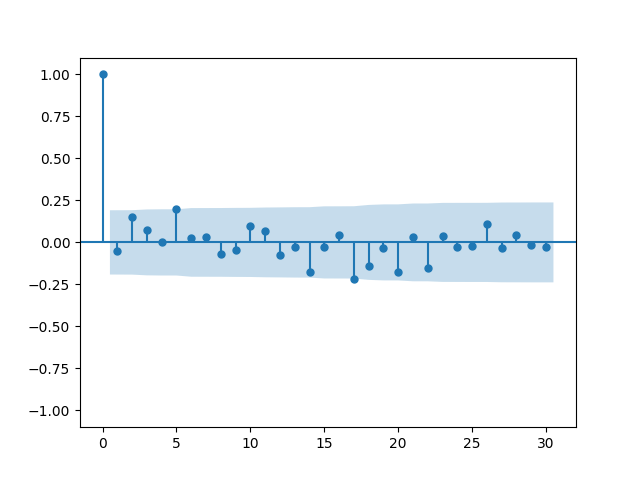}
        \caption{Estonia - residuals ACF}
        \label{fig:sub2}
    \end{subfigure}


        \begin{subfigure}[t]{0.3\textwidth}
        \includegraphics[width=\textwidth]{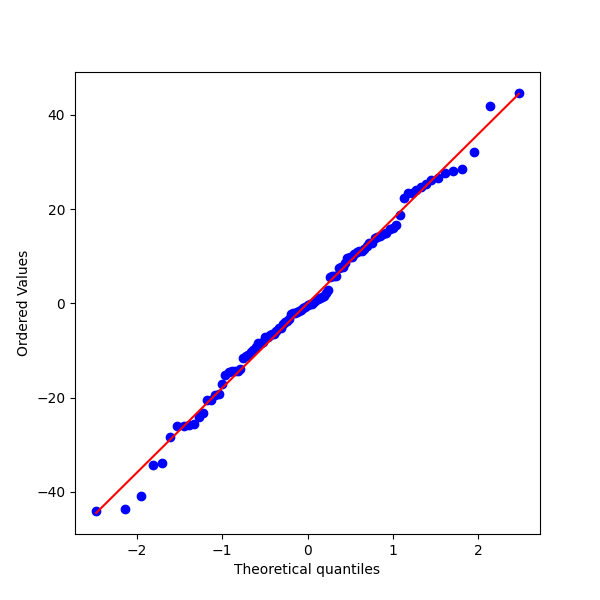}
        \caption{Poland - Q-Q plot of residuals}
        \label{fig:sub6}
    \end{subfigure}
\hspace{0.05\textwidth}
           \begin{subfigure}[t]{0.3\textwidth}
        \includegraphics[width=\textwidth]{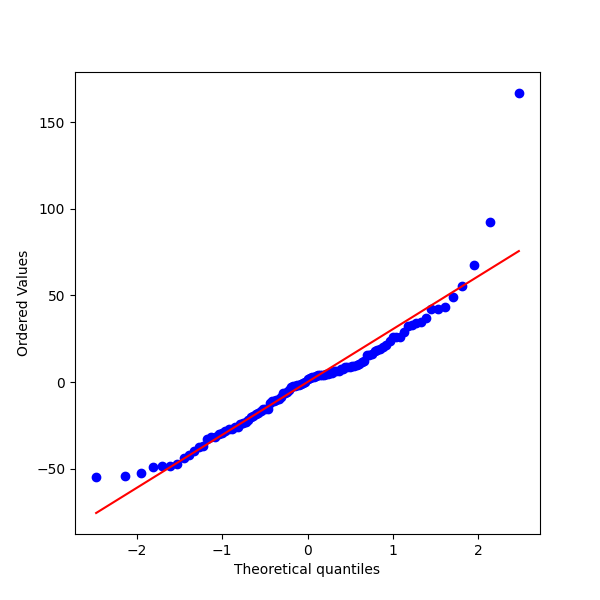}
        \caption{Estonia - Q-Q plot of residuals}
        \label{fig:sub3}
    \end{subfigure}
    \caption{Calibration performance visual tests}
    \label{fig:main}
\end{figure}

\subsection{Implementation of the CfD in Energy Markets}

We now apply the Deep DRBSDE solver to the CfD problem by solving the system \eqref{DRBSDEdecay}, modelling the underlying electricity price with the \( d \)-dimensional OU process whose coefficients are calibrated as detailed in the previous section. This setup enables us to isolate the effects of the penalty structure and barriers within a tractable framework. \\

To model \( X_t \) as in \eqref{eq:OU}, we use the parameters \(\mu\), \(\kappa\), and \(\sigma\) estimated in the calibration section \ref{calibration}, with the values for \(\kappa\) and \(\sigma\) corresponding to the diagonal entries of their respective diagonal matrices, as detailed in Table  \ref{tab:eu-params}. We set the initial price $x$ to a vector of the average historical prices observed for each country within a period used in calibration.

We then address the selection of parameters related to the payoff function \eqref{cfd payoff-explicit} in which we set uniform weights $w_i = \frac{1}{d}, \, i = 1,\ldots,d$. The vector of strike prices $K \in \bR^d$ specifies the agreed-upon exchange prices for electricity, with each component corresponding to a distinct dimension of the forward process. In real world markets, the strike price of a CfD is typically determined through an auction process, reflecting market conditions and competition at the time. However, for the purposes of modelling and to ensure simplicity of presentation, we set the strike price $K$ equal to $\mu + U$, for a $d$-dimensional uniformly distributed random variable $U \sim \text{U}(0.9,1.1)$. Finally, it remains to set the coefficients $\gamma_1$ and $\gamma_2$ in \eqref{barriers}, modelling the penalties for early exits. In the context of hedging risks associated with renewable energy investments,  we assign a higher penalty to Player 1, reflecting the realistic assumption that the regulator (Player 1) is less likely to exit early from a contract within a typically small-sized energy producer (Player 2). We thus set $\gamma_1=1.34$ and $\gamma_2 =0.29$. \\ 

As shown in Figure \ref{fig:loss-cfd}, the loss function for the time steps $n \ge N-3$ decreases across iterations, approaching zero. The losses for $n < N-3$, which are omitted,  exhibit a similar trend. In Figure \ref{fig:Yt-cfd}, we report three realisations for $Y_t$. The estimated distribution of $Y_0$ across $100$ independent simulations is illustrated in Figure \ref{fig:Y0-cfd}. The distribution has mean $1.00$ indicating that the game yields an admissible investment for Player 2, ensuring coverage against potential risk. This suggests that the electricity producer can expect to achieve sufficient returns to secure their position. This outcome aligns with expectations, as the strike price of the game is set above the long-term mean of the price process dynamics.  Additionally, we observed that the empirical distributions of exit times for both the upper and lower barriers are consistent across the $100$ reruns, though these results are omitted.
The histogram in Figure \ref{fig:time-hist-cfd} shows the distribution of the optimal exit times for both players. For the chosen strike price and the penalties, we get that Player 1 early exits the game approximately $8\%$ of the time, and Player 2 $16 \%$ of the time. We also observe that in cases where these decide to exit, Player 1 does so only in the first part of the time horizon and Player 2 in the second. Finally, the scatterplot in Figure \ref{fig:time-scatter-cdf} compares the exit times for each player with the electricity average prices at the time of exit, which slightly decrease for both players as time progresses.  
\begin{table}[H] 
\centering
\begin{tabular}{|l|c|}
\hline
\textbf{Simulation parameter} & \textbf{Value} \\
\hline
\( \text{Time horizon } T \)            & $1.0$            \\
\( \text{Discount factor } \rho \)            & 0.04          \\

\( \text{Upper barrier coeff. } \gamma_1 \)         & 1.34            \\
\( \text{Lower barrier coeff. } \gamma_2 \)         & 0.29           \\
\hline
\end{tabular}
\caption{Parameter values for the CfD example}
\label{table:model_parameters_2}
\end{table}
\begin{figure}[H]
    \centering
    \begin{subfigure}{0.32\textwidth}
        \includegraphics[width=\linewidth]{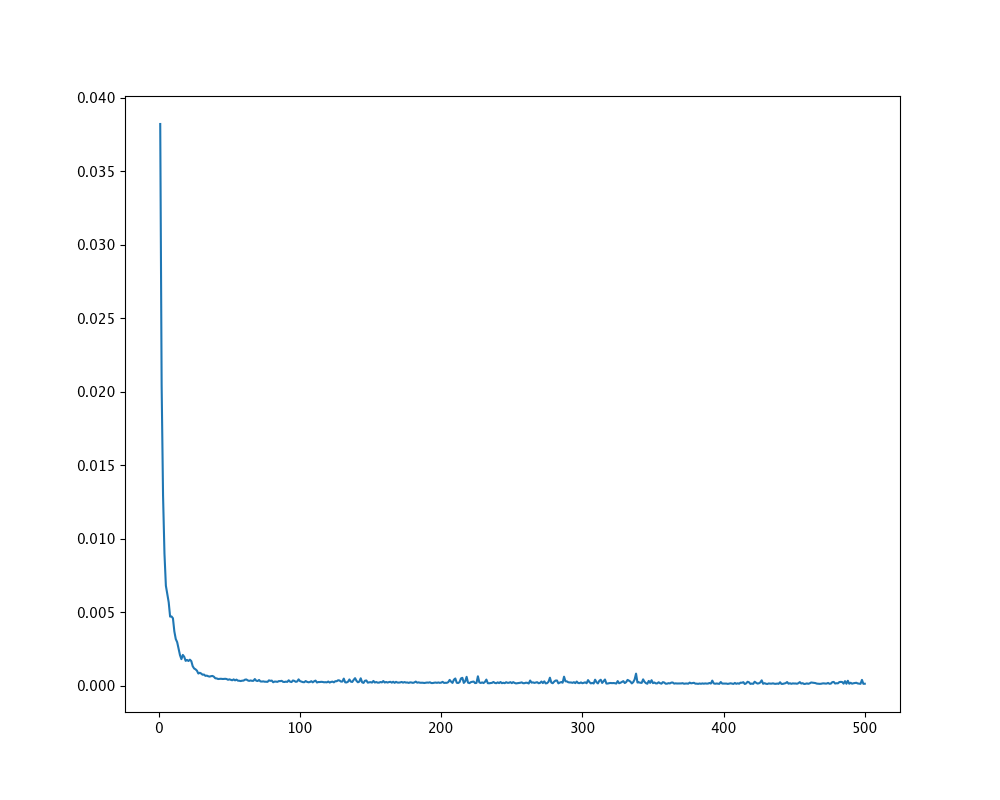}
        \caption{$n =  N-1$}
    \end{subfigure}
    \hfill
    \begin{subfigure}{0.32\textwidth}
        \includegraphics[width=\linewidth]{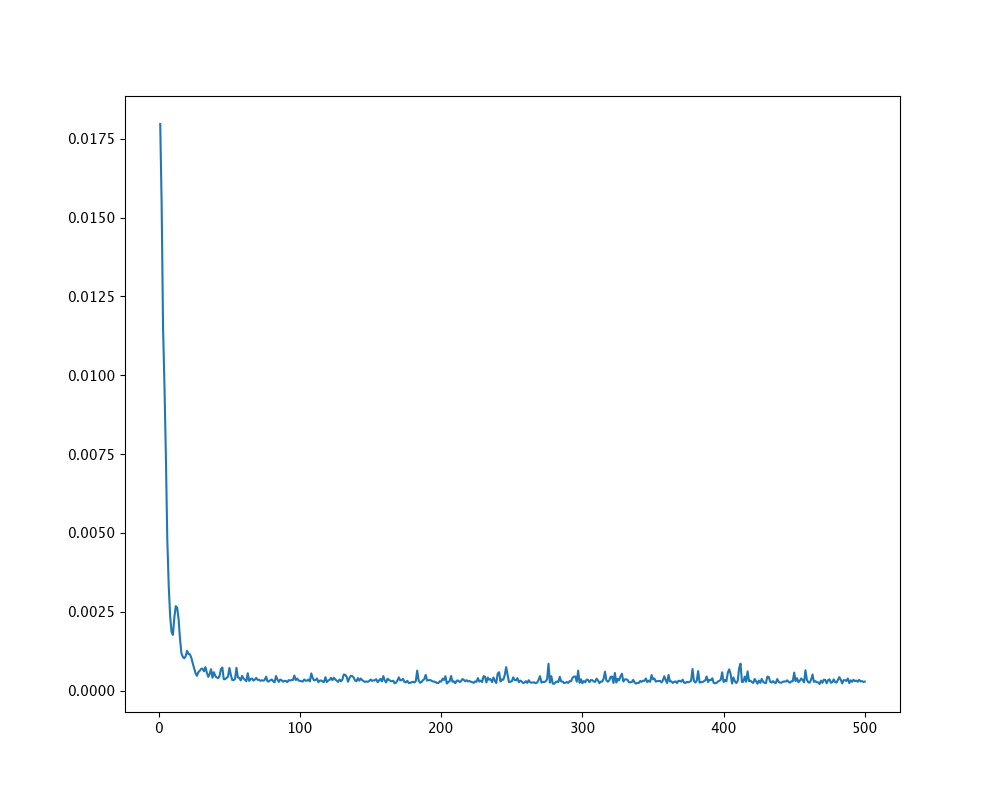}
        \caption{$n =  N-2$}
    \end{subfigure}
    \hfill
    \begin{subfigure}{0.32\textwidth}
        \includegraphics[width=\linewidth]{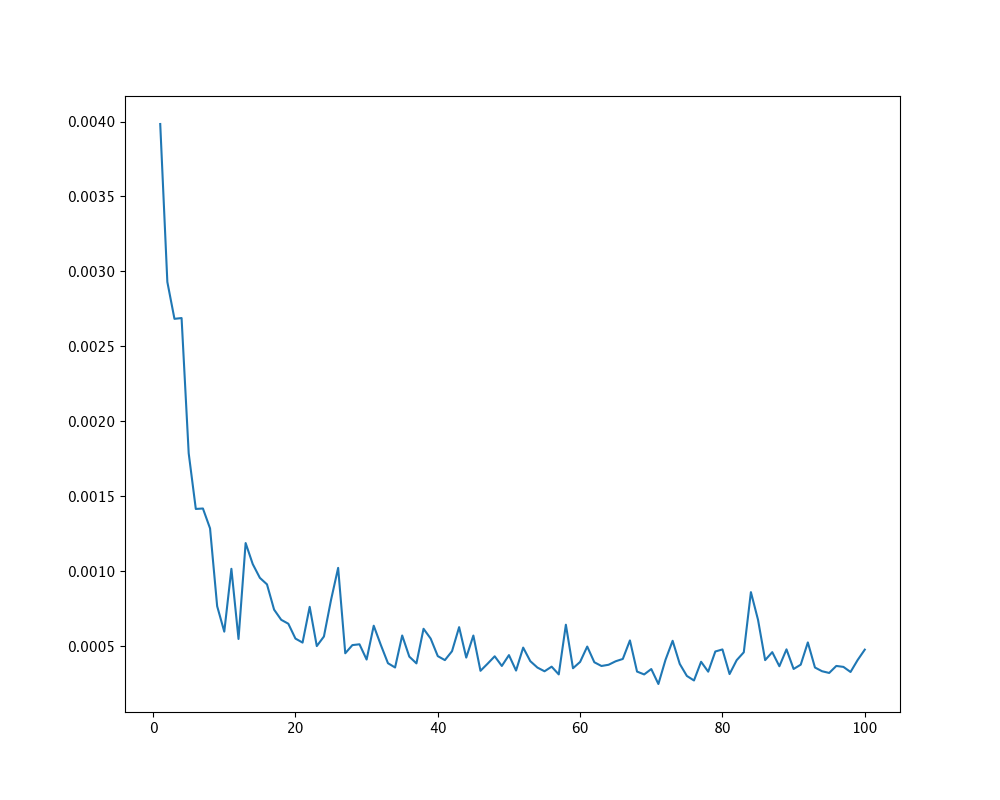}
        \caption{$n =  N-3$}
    \end{subfigure}
    \caption{Convergence of loss function at different timesteps $n$ in the CfD example}
    \label{fig:loss-cfd}
\end{figure}
\begin{figure}[H]
    \centering
    \begin{subfigure}[t]{0.49\textwidth}
        \includegraphics[width=\linewidth]{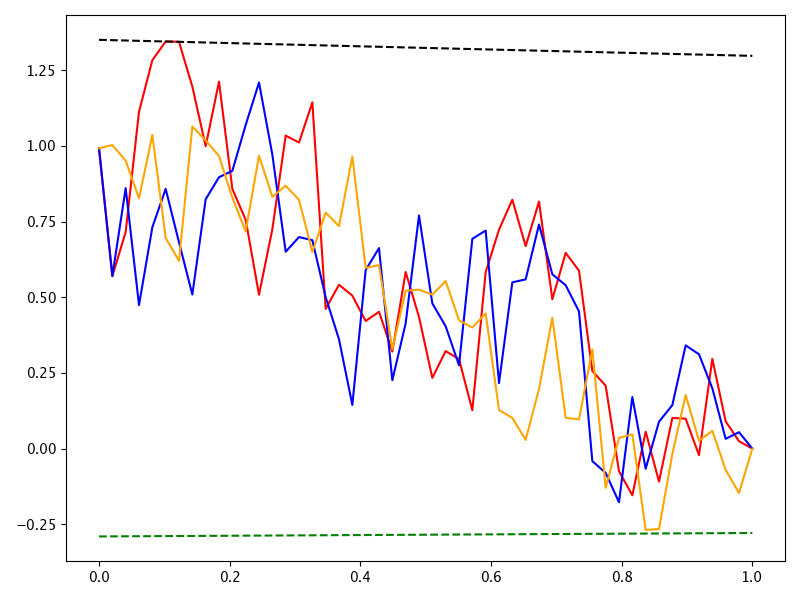}
        \caption{Realizations of $Y_t$}
        \label{fig:Yt-cfd}
    \end{subfigure}
    \hfill
    \begin{subfigure}[t]{0.49\textwidth}
        \includegraphics[width=\linewidth]{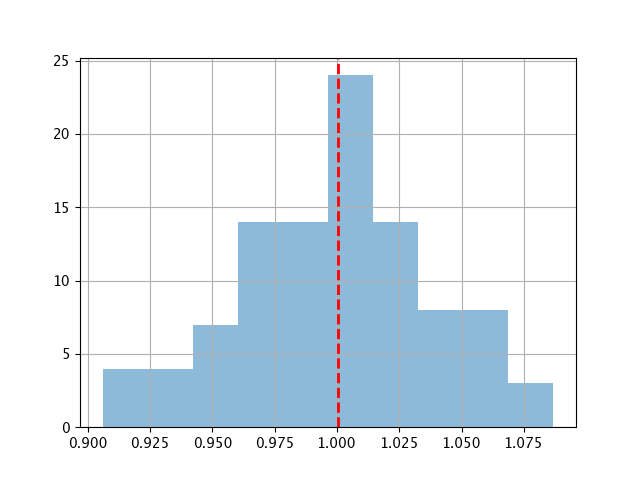}
        \caption{Distribution of estimated $Y_0$}
        \label{fig:Y0-cfd}
    \end{subfigure}
    \caption{Three realisations of the $Y_t$ dynamics and the distribution of estimated $Y_0$ over 100 independent training processes in the CfD example}
    \label{fig:Y-cfd}
\end{figure}
\begin{figure}[H]
    \centering
    \begin{subfigure}[t]{0.49\textwidth}
        \includegraphics[width=\linewidth]{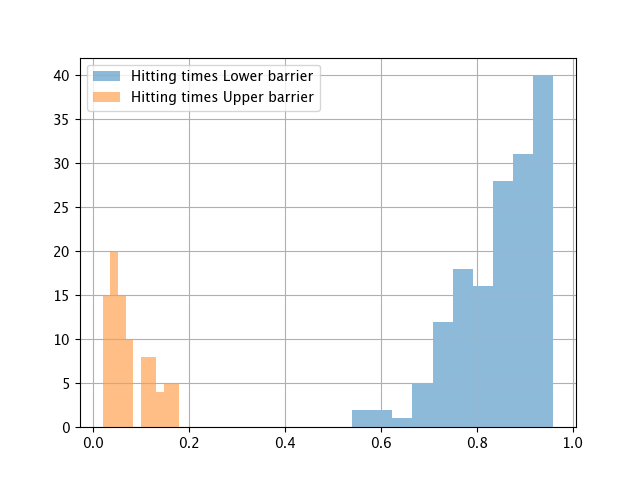}
        \caption{Estimated distribution of the first exit times}
        \label{fig:time-hist-cfd}
    \end{subfigure}
    \hfill
    \begin{subfigure}[t]{0.49\textwidth}
        \includegraphics[width=\linewidth]{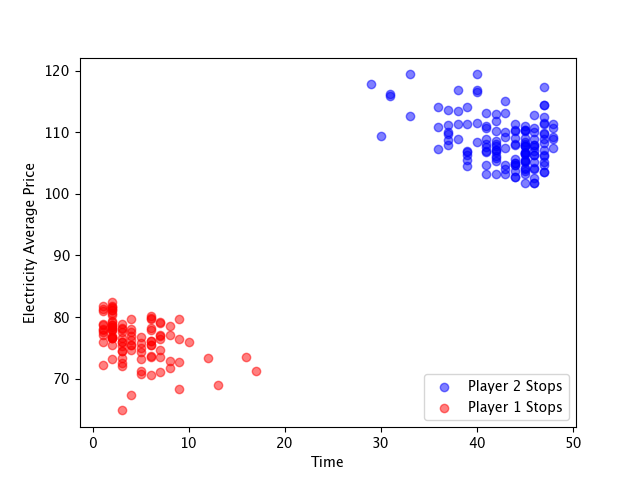}
        \caption{Relationship between first exit times and price dynamics}
        \label{fig:time-scatter-cdf}
    \end{subfigure}
    \caption{First exit time analysis in the CfD example}
    \label{fig:time-cfd}
\end{figure}

\newpage
\section{Conclusion}

This paper introduces a deep learning-based framework for pricing CfDs with early exit clauses, modeled as a two-player zero-sum Dynkin game. The interaction between an electricity producer and a regulator is captured through a stochastic stopping game, where each player can exit early at a cost. We model electricity prices using a mean-reverting OU process and characterize the fair value of the CfD, as well as the optimal exit strategies, through a DRBSDE.

We build on the theoretical correspondence between the value function of the Dynkin game and the solution to a DRBSDE, where the first component represents the contract value and the optimal stopping times correspond to the first hitting times of the barriers. We further introduce a novel representation of the reflection terms as solutions to a Skorokhod problem with time-dependent obstacles, offering structural insight into the dynamics of the cumulative penalties.

To approximate the DRBSDE solution, we design and analyze a deep learning algorithm that approximates both the value function and optimal stopping strategies. We establish convergence results for the method, showing that the approximations become accurate as the network capacity and time discretization improve. The algorithm is validated through numerical experiments, including a symmetric 20-dimensional benchmark game and a realistic CfD model based on 24-dimensional weekly spot electricity price data for continental Europe. These experiments demonstrate the scalability and robustness of the proposed solver in high-dimensional settings, where classical PDE based approaches become infeasible.

The results provide actionable insights for energy market participants and regulators. In particular, we highlight the importance of penalty design and market volatility in shaping early exit behavior and fair contract valuation. The proposed methodology also supports the development of data driven contract design tools that can accommodate complex market structures and risk sharing arrangements.

Looking ahead, we plan to extend the model to incorporate price dynamics with finite variation jumps, which are common in electricity markets due to external shocks and regulatory changes. This extension would enhance the realism and flexibility of the framework and allow for even more accurate modeling of risk and strategic behavior under uncertainty.
\vspace{0.5cm}

\textbf{Acknowledgment.} The authors (N. Agram and G. Pucci) would like to express their gratitude to Ren\'e A\"id for his insightful discussions and valuable feedback during his visit to KTH.

\appendix

\section{Theoretical Background on DRBSDEs and Dynkin Games} \label{appendix:drbsde}
 This appendix collects the theoretical framework supporting the main results. We begin by recalling standard existence and uniqueness results for doubly reflected BSDEs, and then present the abstract Dynkin game formulation and its connection to DRBSDEs. In addition, we include some original results that are specific to the structure and equilibrium representation of the contract studied in the main text.




\subsection{Preliminaries and DRBSDE Framework} \label{sec:preliminaries}

Let $b : [0, T] \times \mathbb{R}^d \to \mathbb{R}^d$ and $\sigma : [0, T] \times \mathbb{R}^d \to \mathbb{R}^{d \times d}$ be measurable functions, $d \ge 1$.
For $(t, x) \in [0, T] \times \mathbb{R}^d$, let $(X_s^{t, x})_{s \in [0, T]}$ be the unique $\mathbb{R}^d$-valued process solution of the following standard SDE:
\begin{equation}
\begin{cases}\label{sde}
    X_s^{t, x} = x + \displaystyle\int_t^s b(r, X_r^{t, x}) \, dr + \int_t^s \sigma(r, X_r^{t, x}) \, dB_r, & t \leq s \leq T, \\
    X_s^{t, x} = x, & s < t.
\end{cases}
\end{equation}
The process $(X_s^{t, x})_{s \in [0, T]}$ represents the underlying asset, which could be, for example, the price process of a financial asset or a commodity.

\begin{assumption} \label{ass1} We assume that the coefficients $b$ and $\sigma$ satisfy the global Lipschitz and linear growth conditions:
\begin{align*}
    |b(t, x) - b(t, y)| + |\sigma(t, x) - \sigma(t, y)| &\leq C_L |x - y|,  \\
    |b(t, x)|^2 + |\sigma(t, x)|^2 &\leq C_L^2(1 + |x|^2),
\end{align*}
for every $0 \leq t \leq T$, $x \in \mathbb{R}^d$, $y \in \mathbb{R}^d$, where $C_L$ is a positive constant. 
\end{assumption}
 It is clear that, under Assumptions \ref{ass1}, the SDE \eqref{sde} has a unique solution (cf. \cite[Theorem 2.9, p. 289]{karatzas}). Moreover, for every $p \geq 2$, there exists $C_p>0$, such that for all $t\in  [0,T]$, 
\begin{equation*}
\mathbb{E}\left[
    \sup_{s \in [t, T]} \left| X^{t, x} _{s} \right|^p \, \middle| \, \mathcal{F}_t
\right] \leq C_p(1 + |x|^p).
\end{equation*}
The constant $C_p$ depends only on the Lipschitz and the linear growth constants of $b$ and $\sigma$.\\

We also introduce the following notations, which will be used throughout the appendix:
\begin{itemize}
    \item $D[0, \infty)$: the space of real-valued RCLL functions on $[0, \infty)$.
    \item $D^{-}[0, \infty)$ (resp. $D^{+}[0, \infty)$): the space of functions in $D[0, \infty)$ taking values in $\mathbb{R} \cup \{-\infty\}$ (resp. $\mathbb{R} \cup \{\infty\}$).
    \item $BV[0, \infty)$: the set of functions in $D[0, \infty)$ of bounded variation on every finite interval.
    \item $I[0, \infty)$: the set of nondecreasing functions in $D[0, \infty)$.
\end{itemize}

We now state the standard formulation of the DRBSDE and its well-posedness under classical assumptions. To define the equation, we consider the following four objects:

\begin{itemize}
    \item[$(i)$] Let $\xi$ be a given random variable in $\mathbb{L}^2$.
    \item[$(ii)$] $ f: [0, T] \times \Omega \times \mathbb{R} \times \mathbb{R}^d \to \mathbb{R}$ be a given $\mathcal{P} \otimes \mathcal{B}(\mathbb{R}) \otimes \mathcal{B}(\mathbb{R}^d)$-measurable function that satisfies
\begin{equation} \label{lipschi}
    \bE \int_0^T f^2(t, \omega, 0, 0) dt < \infty.
\end{equation}
\begin{equation} 
    |f(t, \omega, y, z) - f(t, \omega, y', z')| \leq C (|y - y'| + |z - z'|),
\end{equation}
\[
\forall (t, \omega) \in [0, T] \times \Omega; \quad y, y' \in \mathbb{R}, \quad z, z' \in \mathbb{R}^d
\]
for some $0 < C < \infty$. 
\item[$(iii)$] Consider also two continuous processes $L, U$ in $\mathcal{S}^2$ that are completely separated, i.e., 
\begin{equation} \label{separated}
    L_t <  U_t, \quad \forall 0 \leq t \leq T, \quad \text{and} \quad L_T \leq \xi \leq U_T \quad \text{a.s.}
\end{equation}
\end{itemize}

\begin{definition}
    A solution for the DRBSDE associated with $(f, \xi, L, U)$ is a quadruple of $\mathcal{P}$-measurable processes $\left(Y_t, Z_t, A_t, C_t\right)_{t\in[0,T]}$ from $ S^2 \times \mathcal{H}^{2,d} \times S_{ci} \times S_{ci} $ such that $\mathbb{P}$-a.s.:
\begin{enumerate}
    \item[(i)] For each $t \in [0, T],$
    \begin{equation} \label{doubly ref def}
    Y_t = \xi +  \displaystyle\int_t^T f(r, Y_s,Z_s) ds
   - \int_t^T Z_s dB_s + (A_T - A_t) - (C_T - C_t).
    \end{equation}
    \item[(ii)] $ L_t \leq Y_t \leq U_t$,  $\forall t \leq T$. 
    \item[(iii)] $
    \displaystyle\int_0^T (Y_s - L_s) dA_s = 
    \int_0^T (U_s - Y_s) dC_s = 0$.  
\end{enumerate}
\end{definition}

We have the following existence result (see \cite[Theorem 3.7]{hamadene2005}). 

\begin{theorem} \label{existence of solution DRBSDE}
    Under Assumptions \eqref{lipschi}-\eqref{separated}, there exists a  unique $\mathcal{P}$-measurable process $\left(Y_t, Z_t, A_t, C_t\right)_{t\in[0,T]}$ solution of the DRBSDE \eqref{doubly ref def}.  
\end{theorem}

\subsection{Dynkin Games: Definition and Saddle-Point Structure}\label{sec:dynkin}
\hspace*{0.5 cm} We now introduce the stochastic model that serves as the basis for modeling the contract for differences in the main text. The problem is formulated as a two-player zero-sum Dynkin game, where each player strategically selects an optimal stopping time to maximize their respective payoffs.\\

We consider a zero-sum Dynkin game between two players, Player 1 and Player 2, who are interested in the same asset. The payoff is defined in terms of the underlying diffusion process \eqref{sde}, which models the asset dynamics. The admissible strategies of the players are stopping times with respect to the filtration $\lbrace \mathcal{F}_t \rbrace_{t\geq0}$. 

Let $t\in [0,T]$ be given, and let $\tau_1 \in  \mathcal{T}_{t, T}$ and $\tau_2 \in  \mathcal{T}_{t, T}$ be the stopping times associated with Player 1 and Player 2, respectively. The game between Player 1 and Player 2 is played from time $t$ until $\tau_1 \wedge \tau_2$, where $x \wedge y := \min(x, y)$. During this period, Player 1 pays Player 2 at a random rate $\varphi(s, X^{t,x}_s)$, which depends on both time $s$ and the underlying state process $X^{t,x}_s$.

For some Borel functions $f_1$, $f_2$ and $g$, the payoff structure of the game is defined as follows: If Player 1 exits the game prior to time $T$ and either before or simultaneously with Player 2, i.e., $\tau_1 < T$ and $\tau_1 \leq \tau_2$, Player 1 pays Player 2 an additional amount $f_1(\tau_1, X_{\tau_1})$. Conversely, if Player 2 exits the game first, i.e., $\tau_2 < \tau_1$, Player 2 pays Player 1 an amount $f_2(\tau_2, X_{\tau_2})$. If neither player exits the game before $T$, the game terminates at $\tau_1 = \tau_2 = T$, and Player 1 pays Player 2 a terminal amount $g(X^{t,x}_T)$. The payoff for the Dynkin game on $[t, T]$ is expressed in terms of the conditional expected cost to Player 1, as follows: 
\begin{align} \label{dynkin game}
    J_{t,x}(\tau_1, \tau_2) = \mathbb{E}\bigg[
        \int_t^{\tau_1 \wedge \tau_2} \varphi(s, X_s^{t, x}) \, ds
        &+ f_1(\tau_1, X^{t, x}_{\tau_1}) \mathbbm{1}_{\{\tau_1 \leq \tau_2, \tau_1 < T\}} 
        - f_2(\tau_2, X^{t, x}_{\tau_2}) \mathbbm{1}_{\{\tau_2 < \tau_1\}} \notag \\
        &+ g(X^{t, x}_T) \mathbbm{1}_{\{\tau_1 \wedge \tau_2 = T\}}
        \, \bigg| \, \mathcal{F}_t \bigg], \quad \tau_1, \tau_2 \in \mathcal{T}_{t, T}.   
\end{align}

Notice that the payoff $J_{t,x}(\tau_1, \tau_2)$ is a cost for Player 1 and a reward for Player 2. Therefore, the objective of Player 1 is to choose a strategy $\tau_1 \in \mathcal{T}_{t, T}$ to minimize the expected value $J_{t,x}(\tau_1, \tau_2)$, while Player 2 aims to choose a strategy $\tau_2 \in \mathcal{T}_{t, T}$ that maximizes it. This results in the upper and lower values for the game on $[t, T]$, denoted by $\overline{V}(t,x)$ and $\underline{V}(t,x)$, respectively:
\begin{equation*}
    \overline{V}(t,x) = \essinf_{\tau_1 \in \mathcal{T}_{t, T}} \esssup_{\tau_2 \in \mathcal{T}_{t, T}} J_{t,x}(\tau_1, \tau_2),
    \quad
    \underline{V}(t,x) = \esssup_{\tau_2 \in \mathcal{T}_{t, T}} \essinf_{\tau_1 \in \mathcal{T}_{t, T}} J_{t,x}(\tau_1, \tau_2).
\end{equation*}
    The Dynkin game on $[t, T]$ is considered ``fair'' and is said to have a value if the upper and lower values at time $t$ are equal. This condition can be expressed as:
\begin{equation} \label{dynkin}
    \essinf_{\tau_1 \in \mathcal{T}_{t, T}} \esssup_{\tau_2 \in \mathcal{T}_{t, T}} J_{t,x}(\tau_1, \tau_2) 
    = V(t,x) = \esssup_{\tau_2 \in \mathcal{T}_{t, T}} \essinf_{\tau_1 \in \mathcal{T}_{t, T}} J_{t,x}(\tau_1, \tau_2).
\end{equation}
The shared value, denoted by $V(t,x)$, is referred to as the solution or the value of the game on $[t, T]$. \\

When studying Dynkin games, the first step is to verify whether the game is fair. Subsequently, one seeks admissible strategies for the players that provide the game's value or approximate, i.e., determine whether the game has a saddle point. This leads to the concept of a Nash equilibrium. \\

\begin{definition}[Nash Equilibrium] \label{nash}
A pair of stopping times $(\tau_1^*, \tau_2^*) \in \mathcal{T}_{t, T} \times \mathcal{T}_{t, T}$ is said to constitute a Nash equilibrium or a saddle point for the game on $[t, T]$ if, for any $\tau_1, \tau_2 \in \mathcal{T}_{t, T}$:
\begin{equation*}
    J_{t,x}(\tau_1^*, \tau_2) \leq J_{t,x}(\tau_1^*, \tau_2^*) \leq J_{t,x}(\tau_1, \tau_2^*).
\end{equation*}
\end{definition}
It is straightforward to verify that the existence of a saddle point $(\tau_1^*, \tau_2^*) \in \mathcal{T}_{t, T} \times \mathcal{T}_{t, T}$ ensures that the game on $[t, T]$ is fair, and its value is given by:
\begin{equation*}
    \essinf_{\tau_1 \in \mathcal{T}_{t, T}} \esssup_{\tau_2 \in \mathcal{T}_{t, T}} J_{t,x}(\tau_1, \tau_2) 
    = J_{t,x}(\tau_1^*, \tau_2^*) 
    = \esssup_{\tau_2 \in \mathcal{T}_{t, T}} \essinf_{\tau_1 \in \mathcal{T}_{t, T}} J_{t,x}(\tau_1, \tau_2).
\end{equation*}

Let us now consider the functions 
\[
g : \mathbb{R}^d \to \mathbb{R}, \quad f_1, f_2 : [0, T] \times \mathbb{R}^d \to \mathbb{R}, \quad \varphi : [0, T] \times \mathbb{R}^d \to \mathbb{R},
\]
that satisfy the following assumptions:
\begin{assumption} \label{ass 2}
\begin{enumerate}
    \item Function $g$ is continuous and bounded.
    
    \item Functions $f_1$, and $f_2$ are bounded and continuous. Moreover, for any $(t, x) \in [0, T] \times \mathbb{R}^d$, 
    \begin{equation*}
        -f_2(t, x) < f_1(t, x).
    \end{equation*}
    \begin{equation*}
         -f_2(T, x) \leq g(x) \leq f_1(T, x).
    \end{equation*}
    \item Function $\varphi$ is  $\frac{1}{2}$ H\"older-continuous in $t$,  Lipschitz continuous and has linear growth in $x$. 
\end{enumerate}
\end{assumption} 

 The following theorem shows that the game problem defined above has a value (see e.g., \cite[Theorem 2.1, p. 686]{Peskir}). 

\begin{proposition} \label{value of the game}
Under Assumptions \ref{ass 2}, there exists a continuous $\mathcal{F}_t$-adapted process $(V(t,x))_{0 \leq t \leq T}$ such that for each $t$, the random variable $V(t,x)$ gives the fair value of the Dynkin game on $[t, T]$ assuming $X_t = x$. Furthermore, the debut times $\tau_{2, t}^*$ and $\tau_{1,t}^*$ defined by
\begin{equation*}
    \tau_{2, t}^* := \inf\{s \geq t : V(s,x) = -f_2(s, X_s^{t, x})\} \wedge T,
\end{equation*}
\begin{equation*}
    \tau_{1,t}^* := \inf\{s \geq t : V(s,x) = f_1(s, X_s^{t, x})\} \wedge T,
\end{equation*}
form a saddle point $(\tau_{1,t}^*, \tau_{2,t}^*)$ for the Dynkin game on $[t, T]$.  
\end{proposition}

It is well known that Dynkin game problems are closely linked to BSDEs with two reflecting barriers (see e.g., \cite{cvitanic_karatzas_1996, hamadene, hamadene_hassani}). We further explore this connection in the next section.

\subsection{DRBSDE Representation of Dynkin Games} \label{sec DRBSDEs}

\hspace*{0.5 cm} We now focus on the links between the zero-sum Dynkin game introduced in the last section and the solution of a corresponding DRBSDE with continuous barriers. The following theorem shows that, under Assumption \ref{ass1} and Assumption \ref{ass 2}, the game problem \ref{dynkin game} has a value. Moreover, its value is characterized in terms of the first component of the solution of a DRBSDE (cf., \cite[Theorem 3.8.]{hamadene2005}). 

\begin{theorem} \label{theo:link}
    Under the above assumptions \ref{ass1} and \ref{ass 2}, for any $(t,x)\in [0,T] \times \mathbb{R}^d$, there exists a unique process $\left(Y_s^{t,x}, Z_s^{t,x}, A_s^{t,x}, C_s^{t,x}\right)_{s \leq T}$ $\mathcal{P}$-measurable solution of the DRBSDE associated with 
\begin{equation*}
        \Big(
\varphi( \cdot, X^{t, x}_\cdot), \, g( X^{t, x}_T), \, 
f_1( \cdot, X^{t, x}_\cdot) , \, -f_2( \cdot, X^{t, x}_\cdot ) \Big),
\end{equation*}
 that is,  
\begin{enumerate}
    \item[(i)] $Y^{t,x} \in S^2, \, Z^{t,x} \in \mathcal{H}^{2,d}, \, A^{t,x}\in S_{ci}, \, \text{and} \, C^{t,x} \in S_{ci}$.
    \item[(ii)] For each $s \in [t, T],$
    \begin{equation} \label{doubly ref}
    Y^{t, x}_s = g(X^{t, x}_T)+  \displaystyle\int_s^T \varphi(r, X_r^{t, x}) dr
   - \displaystyle\int_s^T Z^{t, x}_r dB_r + (A^{t, x}_T - A^{t, x}_s)  
    - (C^{t, x}_T - C^{t, x}_s), 
    \end{equation}
    \item[(iii)] $ -f_2(s, X_s^{t, x}) \leq Y^{t,x}_s \leq f_1(s, X_s^{t, x}),$  
    \item[(iv)] $
    \displaystyle\int_t^T (Y^{t, x}_r - (- f_2(r, X_r^{t, x}))) dA^{t, x}_r = 
    \int_t^T (f_1(r, X_r^{t, x}) - Y^{t, x}_r) dC^{t, x}_r = 0.$ \\
\end{enumerate}
 Moreover, the first component of the solution $Y^{t,x}$ is the common value function of the Dynkin game \eqref{dynkin} on $[t , T ]$, i.e.,
 \begin{equation*}
     Y_t^{t,x}= \overline{V}(t,x) = \underline{V}(t,x).
 \end{equation*}
Furthermore, the pair of stopping times $(\tau_{1,t}^*,\tau_{2,t}^*)$ defined by
\begin{equation*}
    \tau_{2,t}^* := \inf\{s \geq t : Y_s^{t,x} = -f_2(s, X_s^{t, x})\} \wedge T,
\end{equation*}
\begin{equation*}
    \tau_{1,t}^* := \inf\{s \geq t : Y_s^{t,x} = f_1(s, X_s^{t, x})\} \wedge T,
\end{equation*}
form a saddle point for the Dynkin game on $[t, T]$. 
\end{theorem}

 \begin{remark}
 This result is classical and not original to this work. The well-posedness of DRBSDEs with Brownian filtration, continuous data, and time dependent barriers was first established in \cite{cvitanic_karatzas_1996} and extended to more general settings (e.g., discontinuous obstacles, irregular filtrations, and Poisson jumps) in \cite{hamadene2005, hamadene_hassani, Ouknine2024, grigorova2018doubly}. We include Theorem~\ref{theo:link} for completeness and to support the Deep DRBSDE  solver construction in Sections \ref{sec:deep_learning}-\ref{sec:num experiments}.
 \end{remark}

Now, notice that given a solution $\left(Y_s^{t,x}, Z_s^{t,x}, A_s^{t,x}, C_s^{t,x}\right)$ to the DRBSDE \eqref{doubly ref} satisfying conditions $(ii)$ to $(iv)$, the problem corresponds, in a deterministic framework, to a Skorokhod problem with two time dependent boundaries. Consequently, by applying some well known properties of the Skorokhod problem, we can derive an explicit formula for the increasing processes $A_s^{t,x}$ and $C_s^{t,x}$.\\

Recall the Skorokhod problem (SP) on a time varying interval $ [\alpha_\cdot, \beta_\cdot]$. 

\begin{definition}(Skorokhod problem)
    Let $\alpha, \beta \in D[0, \infty)$ such that $\alpha \leq \beta$. Given $x \in D[0, \infty)$, a pair of functions $(y, \eta) \in D[0, \infty) \times BV[0, \infty)$ is said to be a solution of the Skorokhod problem on $[\alpha, \beta]$ for $x$ if the following two properties are satisfied:

\begin{enumerate}
    \item[(i)] $y_t= x_t + \eta_t \in [\alpha_t, \beta_t]$, for every $t\geq 0$. 
    \item[(ii)] $\eta(0^-)=0$, and $\eta$ has the decomposition $\eta:= \eta^l - \eta^u$, where $\eta^l, \eta^u \in I[0, \infty)$, 
    \begin{equation*} 
        \displaystyle\int_0^\infty \mathbbm{1}_{ \lbrace y_s < \beta_s \rbrace} d\eta^u_s=  \int_0^\infty \mathbbm{1}_{ \lbrace y_s > \alpha_s \rbrace} d\eta^l_s= 0.
    \end{equation*}
\end{enumerate} 
If $(y, \eta)$ is the unique solution to the SP on $[\alpha_\cdot, \beta_\cdot]$ for $x$, then we will write 
$y = \Gamma_{\alpha, \beta}(x)$, and refer to $\Gamma_{\alpha, \beta}$ as the associated Skorokhod map (SM).
\end{definition}
In  \cite[Theorem 2.6]{burdzy2009skorokhod} found an explicit representation for the so called extended Skorokhod map (ESM), which is a relaxed version of the SP, see Definition 2.2 in \cite{burdzy2009skorokhod}. They show that for any 
$\alpha \in D^{-}[0,\infty)$ and $\beta \in D^{+}[0,\infty)$ such that $\alpha \leq \beta$, 
there is a well-defined ESM $\bar{\Gamma}_{\alpha,\beta}: D[0,\infty) \to D[0,\infty)$ and it is represented by
\begin{equation} \label{formula}
    \bar{\Gamma}_{\alpha,\beta}(x) = x - \Xi_{\alpha,\beta}(x),
\end{equation}
where $\Xi_{\alpha,\beta}(x) : D[0,\infty) \to D[0,\infty)$ is given by
\begin{align*} 
     \Xi_{\alpha,\beta}(x)(t) = \max \Bigg\{ \Big[ (x_0 - \beta_0)^+ 
    \wedge \inf_{0 \leq r \leq t} (x_r - \alpha_r) \Big], \notag\\
    \sup_{0 \leq s \leq t} \Big[ (x_s - \beta_s) \wedge \inf_{s \leq r \leq t} (x_r - \alpha_r) \Big] \Bigg\}. 
\end{align*}
Moreover, if $\inf_{t \geq 0} (\beta(t) - \alpha(t)) > 0$, the ESM $\bar{\Gamma}_{\alpha,\beta}$ can be identified with the SM $\Gamma_{\alpha, \beta}$.

Slaby in \cite{slaby2010explicit} obtained an alternative form of the explicit formula \eqref{formula} that is simpler to understand and that we will use in the following to derive an explicit expression for the processes $A^{t,x}$ and $C^{t,x}$. \\

Let us first introduce the following notations:  for $x_t \in D[0, \infty)$, we denote by $T_\alpha$ and $T_\beta$ the pair of times:
    \begin{align*}
        T_\alpha &:= \min\{s > 0 : \alpha_s - x_s \geq 0\},  \\
        T_\beta &:= \min\{s > 0 : x_s - \beta_s \geq 0\},
    \end{align*} 
    and the functions \begin{align*}
    H_{\alpha,\beta}(x)(t) &= \sup_{0 \leq s \leq t} \left[ \left( x_s - \beta_s \right) 
    \wedge \inf_{s \leq r \leq t} \left( x_r - \alpha_r \right) \right], \\
    L_{\alpha,\beta}(x)(t) &= \inf_{0 \leq s \leq t} \left[ \left( x_s - \alpha_s \right) 
    \vee \sup_{s \leq r \leq t} \left( x_r - \beta_r \right) \right].
\end{align*}

The next result provides an alternative representation formula for \eqref{formula} and corresponds to \cite[Corollary 2.20]{slaby2010explicit}.

\begin{corollary}
Let $\alpha \in D^{-}[0,\infty)$, $\beta \in D^{+}[0,\infty)$ be such that $\inf_{t \geq 0} (\beta(t) - \alpha(t)) > 0$. Then, for every $x \in D[0,\infty)$,
\begin{equation*}
    \Xi_{\alpha,\beta} (x)(t) = \mathbbm{1}_{\{T^\beta < T^\alpha\}} \mathbbm{1}_{[T^\beta, \infty)}(t) H_{\alpha,\beta} (x)(t) 
    + \mathbbm{1}_{\{T^\alpha < T^\beta\}} \mathbbm{1}_{[T^\alpha, \infty)}(t) L_{\alpha,\beta} (x)(t).
\end{equation*}
\end{corollary}

In our setting, the doubly reflected BSDE naturally induces a Skorokhod problem, where the increasing processes \( A^{t,x} \) and \( C^{t,x} \) correspond to the minimal efforts required to keep the value process \( Y^{t,x} \) within the time dependent barriers \( [-f_2(t, X_t^{t,x}), f_1(t, X_t^{t,x})] \). While this connection is implicit in the general theory of DRBSDEs, we provide an explicit and original representation of the reflecting processes in terms of infima and suprema over path functionals. 
\vspace{1em}

The following proposition is, to the best of our knowledge, new. It offers a novel Skorokhod type formulation for the reflection terms \( A^{t,x} \) and \( C^{t,x} \), derived directly from the structure of the DRBSDE solution. This characterization enables a deeper understanding of the behavior of the reflecting terms and may offer further numerical advantages in simulation.

    \begin{proposition}
    \label{theo:sko}
Under Assumptions \ref{ass1} and \ref{ass 2}, let $(t,x)\in [0,T] \times \mathbb{R}^d$. Let  $\left(Y_s^{t,x}, Z_s^{t,x}, A_s^{t,x}, C_s^{t,x}\right)_{t \leq s \leq T}$ be a solution of the DRBSDE \eqref{doubly ref} satisfying conditions $(ii)$ to $(iv)$. Then, for each $s \in[t,T]$,
  \begin{align*}  
    A^{t,x}_s &= \mathbbm{1}_{\{T_2 < T_1\}} \mathbbm{1}_{[T_2, \infty)}(s) 
    \Bigg[
    \inf_{0 \leq r \leq T - s}  \Bigg\lbrace(x_r - f_1(r, X_r^{t, x})) \vee \sup_{r \leq u \leq T - s} (x_u + f_2(u, X_u^{t, x})) \Bigg\rbrace \notag \\
    &- \inf_{0 \leq r \leq T} \Bigg\lbrace (x_r - f_1(r, X_r^{t, x})) \vee \sup_{r \leq u \leq T} (x_u + f_2(u, X_u^{t, x})) \Bigg\rbrace \Bigg],
\end{align*}
\begin{align*}
    C^{t,x}_s &= - \mathbbm{1}_{\{T_1 < T_2\}} \mathbbm{1}_{[T_1, \infty)}(s) 
     \Bigg[ \sup_{0 \leq r \leq T - s} \left\lbrace(x_r + f_2(r, X_r^{t, x})) \wedge \inf_{r \leq u \leq T - s} (x_u - f_1(u, X_u^{t, x}))\right\rbrace 
     \notag \\
    &- \sup_{0 \leq r \leq T}  \Bigg\lbrace (x_r + f_2(r, X_r^{t, x}))\wedge \inf_{r \leq u \leq T} (x_u - f_1(u, X_u^{t, x})) \Bigg\rbrace \Bigg], 
\end{align*}
where
\begin{align*}
        x_s &= g(X^{t, x}_T) + \int_{T - s}^T  \varphi(r, X_r^{t, x}) dr - \int_{T - s}^T Z^{t, x}_r dB_r,  \\
        T_2 &:= \min\{s > t : -f_2(s, X_s^{t, x}) - x_s \geq 0\},  \\
        T_1 &:= \min\{s > t : x_s - f_1(s, X_s^{t, x}) \geq 0\},
    \end{align*} 
\end{proposition}

\begin{proof}
 First, we write the equation \eqref{doubly ref} in its forward form as
\begin{equation*} 
    Y^{t, x}_s = Y^{t, x}_0 -  \displaystyle\int_0^s \varphi(r, X_r^{t, x}) dr
   + \displaystyle\int_0^s Z^{t, x}_r dB_r - K^{t, x}_s, 
    \end{equation*}
where $K^{t,x}_s:= A^{t, x}_s - C^{t, x}_s$. Notice that by Assumption \ref{ass 2}, the pair $(Y^{t, x}_{T-s}, K^{t, x}_{T-s}- K^{t, x}_T)_{0 \leq s \leq T-t}$ solves a Skorokhod problem on $[t,T]$ and the solution pair $(Y^{t, x}_{T-s}, K^{t, x}_{T-s}- K^{t, x}_T)$ can be represented as
\begin{align*}
        K^{t, x}_{T-s}- K^{t, x}_T &= \Xi_{\alpha,\beta} (x_s) \\
        Y^{t, x}_{T-s} &=x_s-\Xi_{\alpha,\beta} (x_s),
\end{align*}
 where
\begin{align*}
     x_s &= g(X^{t, x}_T) + \int_{T - s}^T  \varphi(r, X_r^{t, x}) dr - \int_{T - s}^T Z^{t, x}_r dB_r \\
     \alpha(r)&= -f_2(r, X_r^{t, x})\\
     \beta(r)&=f_1(r, X_r^{t, x}) \\
     \Xi_{\alpha,\beta} (x_s)&= \mathbbm{1}_{\{T_1 < T_2\}} \mathbbm{1}_{[T_1, \infty)}(s) H_{\alpha,\beta} (x)(s) 
    + \mathbbm{1}_{\{T_2 < T_1\}} \mathbbm{1}_{[T_2, \infty)}(s) L_{\alpha,\beta} (x)(s). 
\end{align*}
It follows that 
\begin{align*}
    K^{t, x}_s&= \Xi_{\alpha,\beta} (x_{T-s}) - \Xi_{\alpha,\beta} (x_{T}) \notag\\
    &=  \mathbbm{1}_{\{T_1 < T_2\}} \mathbbm{1}_{[T_1, \infty)}(s) \Big[ H_{\alpha,\beta} (x)(T-s) - H_{\alpha,\beta} (x)(T) \Big] \notag  \\
    &+ \mathbbm{1}_{\{T_2 < T_1\}} \mathbbm{1}_{[T_2, \infty)}(s) \Big[  L_{\alpha,\beta} (x)(T-s) - L_{\alpha,\beta} (x)(T)\Big]\\
    &=A^{t, x}_s - C^{t, x}_s. \notag
\end{align*}
By the Skorokhod condition $(iv)$:
\begin{equation*}
     \displaystyle\int_t^T \Big(Y^{t, x}_r - (- f_2(r, X_r^{t, x}))\Big) dA^{t, x}_r = \int_t^T \Big(f_1(r, X_r^{t, x}) - Y^{t, x}_r\Big) dC^{t, x}_r = 0,
\end{equation*}
we conclude the following: If $T_2 < T_1$, the process $Y_\cdot^{t, x}$ reaches the lower boundary $- f_2(\cdot, X^{t, x})$, so the minimal push $A_\cdot^{t, x}$ is applied to keep the solution inside the two obstacles $- f_2(\cdot, X^{t, x})$ and $f_1(\cdot, X^{t, x})$. Hence,  
\begin{equation*}
    A^{t, x}_s= \mathbbm{1}_{[T_2, \infty)}(s) \Big[  L_{\alpha,\beta} (x)(T-s) - L_{\alpha,\beta} (x)(T)\Big].
\end{equation*}
Conversely, when $T_1 < T_2$, the process $Y_\cdot^{t, x}$ reaches the upper boundary $f_1(\cdot, X^{t, x})$, and the minimal push $C_\cdot^{t, x}$ is applied. Hence, 
\begin{equation*}
    C^{t, x}_s= -\mathbbm{1}_{[T_1, \infty)}(s) \Big[  H_{\alpha,\beta} (x)(T-s) - H_{\alpha,\beta} (x)(T)\Big].
\end{equation*}
The proof is then complete. 
\end{proof}

\bibliography{References}

\end{document}